\documentclass[12pt,centertags,reqno]{amsart}

\usepackage[foot]{amsaddr}
\usepackage{latexsym}
\usepackage[english]{babel}
\usepackage[T1]{fontenc}

\usepackage[utf8]{inputenc}
\usepackage[numbers]{natbib}
\usepackage{amssymb,amsmath}
\usepackage{fancyhdr}
\usepackage{url}
\usepackage{hyperref}
\usepackage{verbatim}
\usepackage{leftidx}
\usepackage{bm}
\usepackage{color,graphicx}
\usepackage{tikz}
\usepackage{bbm}

\usepackage{mathrsfs, mathtools}
\usepackage{stmaryrd}

\usepackage{marvosym}
\mathtoolsset{showonlyrefs}

\usepackage{appendix}

\usepackage{soul}

\usepackage{enumerate} %vari elenchi numerati
\usepackage{enumitem} %vari elenchi puntati

\numberwithin{equation}{section} \makeatletter
\renewcommand{\subsection}{\@startsection
{subsection}{2}{0mm}{\baselineskip}{-0.25cm}
{\normalfont\normalsize\bf}} \makeatother

\newtheorem{theorem}{Theorem}[section]

\newtheorem{definition}[theorem]{Definition}
\newtheorem{remark}[theorem]{Remark}
\newtheorem{proposition}[theorem]{Proposition}
\newtheorem{example}[theorem]{Example}
\newtheorem{assumption}[theorem]{Assumption}

\def \F {\mathcal F}

\def \P {\mathbf P}

\def \R {\mathbb R}
\def \bF {\mathbb F}

\def \bE {\mathbb E}

\newcommand{\ud}{\mathrm d}

\DeclareMathOperator*{\essinf}{ess\,inf}
\DeclareMathOperator*{\esssup}{ess\,sup}

\def \a {{(1)}}
\def \b {{(2)}}

\def \a {{(1)}}
\def \b {{(2)}}

\begin{document}
\author[C.~Ceci]{Claudia  Ceci}
\address{Claudia  Ceci, Department MEMOTEF,
University of Rome Sapienza, Via del Castro Laurenziano, 9,
Rome, Italy.}\email{claudia.ceci@uniroma1.it}
\author[A.~Cretarola]{Alessandra Cretarola\, %\orcidlink{0000-0003-1324-9342}
}
\address{Alessandra Cretarola, Department of Economic Studies,
 University "G. D'Annunzio'' of Chieti-Pescara,
Viale Pindaro, 42, I-65127 Pescara, Italy.}\email{alessandra.cretarola@unich.it}

\title{Self-protection and self-insurance for general risk models via a BSDE approach}

                               %%
%%%%%%%%%%%%%%%%%%%%%%%%%%%%%%%%%%%%%%%%%%%%%%%%%%%%%%%%%%%%%%%%%%%

\begin{abstract}
We investigate an optimal prevention and insurance problem in a general risk setting, where a representative agent is exposed to potential losses. The agent adopts a strategy that combines self-protection, aimed at reducing the frequency of claims, and self-insurance, aimed at mitigating their severity.
The problem which consists in maximizing the expected exponential utility of terminal wealth is formulated as a stochastic control problem and solved by means of  backward stochastic differential equations (BSDEs).  
Our approach, essentially based on a general Bellman Optimality  Principle (see \cite{ek81} among others), does not require specification of the underlying filtration structure, making it applicable to a broad class of risk models, including Markov-modulated, stochastic factor, Cox-shot noise and self-excited models. We extend recent results by 
\cite{bhk2023,BCCS_2025}, which focused on self-protection in specific models, by allowing for both self-protection and self-insurance within a unified and general framework.
    \end{abstract}

\maketitle



%60G57: random measures
%60J75: jump processes
%93E20: stochastic control
%91B30: risk theory, insurance
%65C30: Stochastic differential and integral equations 

% REQUIRED
%60G57: random measures
%60J75: jump processes
%93E20: stochastic control
%91B30: risk theory, insurance
%65C30: Stochastic differential and integral equations
\section{Introduction}
Self‐protection, or primary prevention, comprises proactive interventions, such as reinforcing riverbanks against flooding, installing smoke detectors to avert fire damage, or adopting healthier lifestyles to forestall illness, that serve to lower the probability of an adverse event. Instead, self‐insurance, or secondary prevention, does not seek to prevent the event itself but rather aims to reduce the severity of the loss, and then mitigate its financial consequences: for instance, by selecting higher deductibles, maintaining dedicated contingency reserves, or purchasing retention layers in one’s insurance policy. Strong prevention efforts can reduce the need for large financial reserves, but some risks, like major earthquakes, cyber-attacks, or pandemics, can never be entirely avoided. In those cases, primary and secondary prevention interact closely, making it essential for individuals, firms, and policymakers to understand their trade-offs and synergies when allocating limited resources. 
\\
As observed in \cite{bhk2019}, the first systematic study of prevention strategies, whether via self-insurance or self-protection, appears in the seminal paper by \cite{eb72}. Their comparative-statics analysis yields two key conclusions:

(A) Market insurance and self-insurance are substitutes: a higher insurance price prompts greater prevention effort to reduce loss size.

(B) Market insurance and self-protection can be complements: a more expensive insurance premium leads to less effort devoted to lowering the probability of loss.

We also remark that under distortion risk measures in \cite{czz2025} the authors derive explicit formulas for the optimal retention level and prevention intensity, showing that both controls increase monotonically as the market-insurance loading rises, thus confirming (A) and, with respect to (B), finding no evidence of complementarity: prevention effort never falls as premiums become more expensive. While the authors in \cite{bhk2023,BCCS_2025}  focus on a static market insurance and dynamic self-protection actions,  this paper  replaces  static market insurance with dynamic self-insurance strategies. \\ 
Precisely, in this paper we investigate an optimal prevention–insurance problem in a fully general risk framework, where an agent dynamically allocates efforts between reducing claim frequency with self-protection actions and limiting claim severity by applying self-insurance strategies. For example, flood risk can be mitigated by reinforcing levees (self-protection) and reducing exposure in high-risk zones (self-insurance), while fire risk combines the installation of smoke detectors with fire-resistant construction and suppression systems.
By combining self-protection and self-insurance in a unified control model, we capture the trade-offs that arise in applications ranging from flood management and cybersecurity to public health planning. Our setting accommodates a broad class of risk models characterized by stochastic claim arrival intensities, including Markov-modulated processes, stochastic-factor models, Cox shot-noise and self-exciting processes, and permits both controls to adapt over time to evolving risk exposures.
This contrasts with the vast majority of the literature on prevention, where the loss process is typically modeled as a compound Poisson process (\cite{Loisel1, Loisel2,bhk2019}).
\\
The problem under consideration, which consists in maximizing the expected exponential utility of terminal wealth, is formulated as a stochastic control problem and solved by means of  backward stochastic differential equations (BSDEs).  Our BSDE approach rests on a general Bellman optimality principle (see \cite{ek81}  among others) and  does not require specification of the underlying filtration structure.  By avoiding reliance on strict Markovianity or differentiability of the value function, this methodology accommodates an extensive range of risk dynamics.\\
Previous contributions by \cite{bhk2023,BCCS_2025} employed a BSDE approach to address the exponential utility maximization problem under pure self-protection within specific model settings where a martingale representation theorem holds.
We extend their results to a fully general risk framework in which the decision maker simultaneously controls both claim‐arrival intensity and claim severity. 
We consider a general filtration by incorporating a term in the BSDE solution that is a martingale orthogonal to the compensated integer-valued martingale measure associated with the cumulative loss process. The study of BSDEs driven by a general martingale under a general filtration, without assuming the martingale representation property,  was initiated in the seminal work of \cite{ElKaroui_Huang}, and has since been extended in various directions, including the recent contributions by \cite{PPS2018}.\\
 Under minimal regularity conditions, we prove a verification theorem, see Theorem \ref{VerThm}, and establish the existence of the BSDE solution in Theorem \ref{thm:bsde} by relying on the theory of BSDEs that satisfies stochastic Lipschitz conditions (see e.g. \cite{ElKaroui_Huang} and \cite{PPS2018}). Furthermore, we discuss the pure self-protection case and show how our framework both recovers and generalizes the results of \cite{bhk2023,BCCS_2025}. Finally, we
derive more explicit characterizations of the optimal joint self‐protection and self‐insurance strategy in a Markovian framework.
Moreover, a closed formula for the value function in the case of compound Poisson process is provided.\\
The remainder of the paper is organized as follows. In Section \ref{sec:model} we introduce the general loss‐dynamic framework, define admissible prevention efforts, and formulate the exponential‐utility maximization problem. Section \ref{sec:BSDE} discusses the stochastic control problem via BSDEs, proves a verification theorem and existence of a unique solution, and identifies some key classes of risk models that satisfy the assumptions of our main result. Section \ref{sec:self-prot} is devoted to the pure self‐protection case. 
Section \ref{sec:markov} focuses on the Markovian setting, which enables an explicit characterization of optimal controls, 
in particular under bounded claim arrival intensity. 
Finally, the proof of Theorem  \ref{thm:bsde} can be found in Appendix \ref{app:proof}.

\section{Modeling framework}\label{sec:model}

Let $( \Omega , \mathcal{F}, \P^0 ; \mathbb{F} )$ be a filtered probability space and assume that the filtration $\mathbb{F} =\{\mathcal{F} _t, \ t \in [0,T] \}$ satisfies the usual hypotheses. Here, $T>0$ is a finite time horizon. 

An agent facing potential losses can take proactive measures by adopting a prevention strategy aimed at mitigating both the frequency and severity of claims. The agent's actions are modeled through a prevention effort, represented as an $\mathbb{F}$-predictable bidimensional process $\mathbf u=\{(u_t^\a, u_t^\b),\ t \in [0,T]\}$, taking values into the set $U:=[0, \zeta_1] \times [0,\zeta_2]$, where $\zeta_i > 0$, $i=1,2$, are fixed parameters.%, representing the maximum effort. 

Specifically, the proposed model includes two intervention strategies: the first focuses on reducing the intensity and frequency of claims, while the second aims to lower the size of the claims.

Prevention efforts come with two monetary cost functions $c_1$ e $c_2$, that are assumed to be continuous and increasing in $[0,\zeta_1]$ and $[0,\zeta_2]$, respectively, and satisfying $c_1(0)=c_2(0)=0$. 
The agent can also invest in a bank account with constant interest rate $r \in \R^+$ and determines her/his level of effort by maximizing her/his expected utility at terminal time $T$. 
The cumulative losses are described by the pure jump process
$J=\{J_t,\ t\in [0,T]\}$ defined as
\begin{equation*}
    J_t = \sum_{i=1}^{N_t} Z_i,\quad t \in [0,T],
\end{equation*}
where $\{Z_i\}_{i \geq 1}$ is a sequence of i.i.d. random variables with a cumulative distribution function $F: (0,+\infty) \to [0,1]$  satisfying $\bE^{\P^0}[Z^2]<+\infty$, and $N=\{N_t,\ t \in [0,T]\}$, with $N_t = \sum_{i \geq 1}  \mathbbm{1}_{( T_i \leq t) }$ being the claim counting process. 
Let $m(\ud t,\ud z)$ denote the jump measure associated to $J$, given by:
\begin{equation*}\label{jump_measure}
m (\ud t,\ud z) = \sum_{i \geq 1}  \delta_{(T_i, Z_i)} (\ud t,\ud z)  \mathbbm{1}_{( T_i < +\infty) }, 
\end{equation*}
where $\delta_{(t,z)}$ denotes the Dirac measure in $(t,z)$. The marked point process $J$ can be written also as follows:
\begin{equation*}\label{jump_dynamics}
J_t = \int_{0}^{t} \int_{0}^{+\infty } z m(\ud s,\ud z), 
\qquad t\in[0,T].
\end{equation*}
We propose a modeling framework where the agent undertakes self-protection activities, which means that her/his actions reduce the intensity and the severity of the claims. 
Under a null effort the claim arrival process $N$ has $\mathbb F$-predictable stochastic intensity 
$\{\lambda^0_t,\ t \in [0,T]\}$ and the claim size distribution function is $F: [0,+ \infty) \to [0, 1]$. We assume that $\{\lambda^0_t,\ t \in [0,T]\}$ is a strictly positive $\mathbb F$-predictable process satisfying the following condition.
\begin{assumption}\label{ass:epsilon}
     There exists $\epsilon >0$ such that 
\begin{equation*}\label{eq:epsilon}
\mathbb{E}^{\P^0}\left[ \exp{\left((1+\epsilon)\int_0^T \lambda^0_t \ud t\right)}\right]<+\infty.   
\end{equation*}  
\end{assumption}
 Here, $\mathbb{E}^{\P^0}[\cdot]$ denotes the expectation under $\P^0$. Under $\P^0$,  the dual-predictable projection (or compensator measure)   of $m(\ud t,\ud z)$ is given by
\begin{equation*}\label{comp-zero}
\nu^0(\ud t,\ud z) = \lambda^0_{t} F(\ud z)\ud t, 
\end{equation*}
and we denote by $\widetilde m^0(\ud t,\ud z):= m(\ud t,\ud z) - \nu^0(\ud t,\ud z) $ the $\P^0$-compensated random measure.
Let $\mathbf u=\{(u_t^\a, u_t^\b),\ t \in [0,T]\}$ be a prevention strategy, where the first line of intervention aims at reducing  the claim arrival intensity. Precisely, for any $t \in [0,T]$
$$ \lambda^{\mathbf u}_t = \gamma^\a (u_t^\a) \lambda^0_t,
$$
where $\gamma^\a(u^\a)$ takes values in $(0,1]$, is non-increasing and continuous function on $u^\a\in [0, \zeta_1]$ and $\gamma^\a(0) = 1$.
According to classical literature on stochastic control (see, for instance \cite{ek81}) the choice of a dynamic stochastic control acts on the law of the loss process $J$ and we will introduce a new probability measure $\P^{\mathbf u}$ equivalent to $\P^0$  such that the jump measure $m(\ud t,\ud z)$ has controlled compensator measure
\begin{equation}\label{comp}
\nu^{\mathbf u}(\ud t,\ud z)= \gamma^\a(u_t^\a) \nu^0(\ud t,\ud z)= \gamma^\a(u_t^\a) \lambda^0_{t} F(\ud z)\ud t
\end{equation}
where $\gamma^\a(u^\a)$ takes values in $(0,1]$, is non-increasing and continuous function on $u^\a\in [0, \zeta_1]$ and $\gamma^\a(0) = 1$.

The probability measure $\P^{\mathbf u}$  will be defined through a change of probability measure on the space $( \Omega , \mathcal{F}, \P^0 ; \mathbb{F} )$. To this end, we introduce the Dolèans-Dade exponential  $L^{\mathbf u}=\{L_t^{\mathbf u},\ t \in [0,T]\}$ 
$$
L_t^{\mathbf u}:= \mathcal{E}\left (\int_{0}^{t}\int_0^{+\infty} (\gamma^\a(u_s^\a)- 1) \widetilde m^0(\ud s,\ud z) \right ),
$$ 
which is explicitly given by  
\begin{align}
    L_t^{\mathbf u}&=\exp { \left( - \int_{0}^{t}\int_0^{+\infty} (\gamma^\a(u_s^\a)- 1) \nu^0(\ud s,\ud z) + \int_{0}^{t} \int_0^{+\infty}\!\!\log{(\gamma^\a(u_s^\a))} m(\ud s,\ud z) \right) }\nonumber\\
    & =\exp { \left( - \int_{0}^{t} (\gamma^\a(u_s^\a)- 1) \lambda_s^0\ud s + \int_{0}^{t} \int_0^{+\infty}\!\log{(\gamma^\a(u_s^\a))} m(\ud s,\ud z) \right) }.\label{def:elle}
\end{align}

\begin{proposition}\label{Girsanov}
  Under Assumption \ref{ass:epsilon}, for any $\bF$-predictable process taking value in $U$, $\mathbf u=\{(u^\a_t, u^\b_t),\ t \in [0,T]\}$, the  Radon-Nikod\'ym  density process $L^{\mathbf u}$ given in \eqref{def:elle} is an $(\mathbb F,\P^0)$-martingale. 
\end{proposition}

\begin{proof} First we observe that $L^\mathbf u$ is a positive $(\mathbb F,\P^0)$-local martingale because it is the Dolèans-Dade exponential of an $(\mathbb F,\P^0)$-martingale. In fact, since for any $u^\a \in  [0, \zeta_1]$, $0 <\gamma^\a(u^\a)\leq 1$, 
$$
\mathbb{E}^{\P^0}\Big [\int_0^T\int_0^{+\infty} |\gamma^\a(u_s^\a)- 1| \nu^0(\ud s,\ud z) \Big ] \leq \mathbb{E}^{\P^0}\Big [\int_0^T \lambda^0_t \ud t \Big]$$
which is finite thanks to  Assumption \ref{ass:epsilon}.
Note that $0 <\gamma^\a(u^\a)\leq 1$ implies $\log{(\gamma^\a(u^\a))} <0 $ and so by \eqref{def:elle} 
\begin{equation}
\label{stima}
 L^{\mathbf u}_T \leq \exp \left ({  - \int_{0}^{T} (\gamma^\a(u_t^\a)- 1) \lambda_t^0\ud t} \right ) \leq \exp{\left(\int_0^T \lambda^0_t \ud t\right)}. 
 \end{equation}
Finally, by Assumption \ref{ass:epsilon}, the process
$L^{\mathbf u}$ turns to be a uniformly integrable $(\mathbb F,\P^0)$-supermartingale,
hence a true $(\mathbb F,\P^0)$-martingale. Indeed
\begin{align}
\mathbb{E}^{\P^0}\left[ \left(\sup_{t\in[0,T]} L^{\mathbf u}_t\right)^{1+\epsilon}\right] & \leq \mathbb{E}^{\P^0}\left[\exp{\left((1+\epsilon)\int_0^T \lambda^0_t \ud t\right)}\right]<+\infty. %\notag \qedhere
\end{align}
\end{proof}

\noindent Thanks to Proposition \ref{Girsanov}, we define the probability measure $\P^{\mathbf u}$ on $(\Omega, \mathcal{F} _T)$ as
\begin{equation*}
    \frac{\ud \P^{\mathbf u} }{\ud \P ^0}\Big{|}_{\F_T} = L_T^{\mathbf u}.
\end{equation*} 
By the Girsanov Theorem we have that under $\P^{\mathbf u}$ the integer-valued random measure $m(\ud t, \ud z)$ has dual predictable projection given by \eqref{comp} and we denote by \begin{equation*}
\widetilde m^{\mathbf u}(\ud t, \ud z) := m(\ud t, \ud z) - \nu^{\mathbf u}(\ud t,\ud z) = m(\ud t, \ud z)- \gamma^\a(u_t^\a) \lambda^0_{t} F(\ud z)\ud t\end{equation*} 
the compensated random measure. 

Let $\mathbf u=\{(u_t^\a, u_t^\b),\ t \in [0,T]\}$ be a prevention strategy, the second line of interventions $u^\b$ aims at reducing claim sizes. Precisely, under the action of the effort $\mathbf u=\{(u^\a_t, u^\b_t),\ t \in [0,T]\}$, the cumulative losses covered by the agent until time $t$ are
\begin{equation*}
    J^{\mathbf u}_t = \sum_{i=1}^{N_t} \gamma^\b(u^\b_{T_i}) Z_i = \int_{0}^{t} \int_{0}^{+\infty }  \gamma^\b(u^\b_s) z m(\ud s,\ud z), 
\quad t\in[0,T]
\end{equation*}
 where $\gamma^\b(u^\b)$ takes values in $[0,1]$, is non-increasing and continuous function on $u^\b\in [0,\zeta_2]$ and $\gamma^\b(0) = 1$ and,  we recall, under $\P^{\mathbf u}$, $m(\ud s,\ud z)$ has controlled predictable compensator given by \eqref{comp}.

The agent can invest in a bank account with constant interest rate $r \in \R^+$, her/his wealth  
associated to the pair $\mathbf{u}= (u^\a,u^\b)\in \mathcal{U} $ (here, $\mathcal{U}$ denotes the class of admissible efforts which will be defined below in Definition \ref{def:adm_UU}) satisfies
\begin{equation}\label{Wealth_Dynamics_new}
\ud X^{\mathbf u}_t = r X^{\mathbf u}_t \ud t - (c_1 (u_t^\a) + c_2 (u_t^\b) )  \ud t - \int_{0}^{+\infty } \gamma^\b(u_t^\b) z m(\ud t,\ud z), 
\end{equation}
with initial condition $X^{\mathbf u}_0=x_0 \in \R^+$.
The explicit solution to equation \eqref{Wealth_Dynamics_new} is given by
\begin{align} 
X^{\mathbf u}_t = & x_0  e^{rt} - \!\int_0^t e^{r(t-s)}(c_1(u_s^\a) + c_2(u_s^\b) )\ud s + \nonumber\\
& -   \int_0^t\int_{0}^{+\infty } \!\!\!\gamma^\b(u_s^\b) e^{r(t-s)} z m(\ud s,\ud z) \label{Wealth_Dynamics_sol}
\end{align}
for each $t \in [0,T]$, and the discounted wealth $\bar{X}^{\mathbf u}=\{\bar{X}^{\mathbf u}_t,\ t \in [0,T]\}$ with $\bar{X}^{\mathbf u}_t := e^{-rt} X^{\mathbf u}_t$ is given by
\begin{equation*}\label{discount_W}
\bar{X}^{\mathbf u}_t = x_0 - \int_0^t e^{-rs} ( c_1(u_s^\a) + c_2(u_s^\b)) \ud s - \int_0^t\int_{0}^{+\infty } \gamma^\b(u_s^\b) e^{-rs} z m(\ud s,\ud z).
\end{equation*}
In order to formulate our optimization problem we introduce the exponential utility function:
\begin{equation*}\label{exp_utility}
U(x) = 1 - \exp{(-\eta x)},\quad x\in \mathbb{R},
\end{equation*}
where $\eta \in \R^+$ is the constant relative risk aversion coefficient.
The agent aims at maximizing its expected utility at terminal time $T>0$, which is equivalent to
\begin{equation}\label{optimal_problem}
v:=\inf_{{\mathbf u}\in \mathcal{U}}   
\mathbb{E}^{\P^{\mathbf u} } \left[ e^{- \eta X^{\mathbf u}_T} \right].
\end{equation}

\begin{definition}
\label{def:adm_UU}
We define by $\mathcal{U}$ the class of {\em admissible prevention efforts}, which are all the $U=[0, \zeta_1]\times[0,\zeta_2]$-valued and $\mathbb{F}$-predictable processes ${\mathbf u}=\{(u^\a_t, u^\b_t),\ t \in [0,T]\}$,  such that 
\begin{equation*}
{\mathbb{E}^{\P^{\mathbf u}}\left[e ^{-\eta X^{\mathbf u}_T}\right]} < +\infty.
\end{equation*}
Given $t \in [0,T]$, we will denote by $ \mathcal{U}_t$ the class $\mathcal{U}$ restricted to the time interval $[t,T]$.
\end{definition}

\begin{proposition}\label{Adm}
Under the following integrability conditions
\begin{equation}\label{Ass2}
\mathbb{E}^{\P^0}\left[ e^{2\int_0^T \lambda^0_t \ud t}\right]<+\infty, \quad \mathbb{E}^{\P^0}\left[ e^{2 \eta e^{rT}J_T}\right]<+\infty,
\end{equation}
any $U$-valued, $\mathbb{F}$-predictable process is an admissible prevention effort.
\end{proposition}
\begin{proof}
Let $\mathbf{u}=\{(u^\a_t, u^\b_t),\ t \in [0,T]\}$ be a $U$-valued, $\mathbb{F}$-predictable process.
In view of \eqref{def:elle} and \eqref{stima}, since $c^\a(u^\a)$ and $c^\b(u^\b)$ are bounded functions, $\gamma^\b(u^\b)\leq 1$ and using the inequality $ab \leq \frac{1}{2}(a^2+b^2)$, for every $a,b \in \R$, we have that  
\begin{align}
& {\mathbb{E}^{\P^{\mathbf u}}\left[e ^{-\eta X^{\mathbf u}_T}\right]}  = \mathbb{E}^{\P^0}\left[ L_T^\mathbf u e ^{-\eta X^\mathbf u_T}\right]\nonumber\\
& \leq \mathbb{E}^{\P^0}\!\!\left[e^{\int_0^T \lambda^0_t \ud t} e ^{-\eta \left(x_0e^{rT} - \int_0^T e^{r(T-s)}(c_1(u^\a_s ) + c_2(u^\b_s) )\ud s - \int_0^T\int_{0}^{+\infty }\! \gamma^\b(u_s^\b) e^{r(T-s)} z m(\ud s,\ud z)\right)}\right]\nonumber\\
& \leq C \mathbb{E}^{\P^0}\!\!\left[e^{\int_0^T \lambda^0_t \ud t} e^{\eta e^{rT}\int_0^T\int_{0}^{+\infty }z m(\ud s,\ud z)}\right] \nonumber \\
& \leq \frac{C}{2}\left(\mathbb{E}^{\P^0}\!\!\left[e^{2\int_0^T \lambda^0_t \ud t}\right]+\mathbb{E}^{\P^0}\!\!\left[e^{2 \eta e^{rT}J_T}\right]\right) \label{stima_sum}
\end{align}
for some suitable constant $C>0$, and \eqref{stima_sum} is finite in view of \eqref{Ass2}.
\end{proof}

\section{The optimization problem with a BSDE approach} \label{sec:BSDE}

In this section we discuss the stochastic control problem \eqref{optimal_problem} by means of BSDEs.
The BSDE approach transcends the classical method  based on the Hamilton–Jacobi–Bellman equation, requiring neither the Markov property nor, in Markovian cases, any regularity of the value function, and thus applies equally to non-Markovian settings, partial-information models in the presence of infinite-dimensional filter processes 
(see, e.g., \cite{elkaroui-peng-quenez1997, delong2013, ceci-gerardi2011, lim-quenez2015, ceci-colaneri-cretarola2020} and references therein).
Moreover, in contrast to much of the existing literature in the insurance framework (see, e.g. \cite{bhk2023, BCCS_2024, BCCS_2025, BC_2020, ceci-cretarola2025, delong2013}), our approach does not require specifying the underlying filtration $\mathbb{F}$ and can be applied to a broad class of risk models (see Subsection \ref{Exa}).
Taking \eqref{Wealth_Dynamics_sol} into account, problem  \eqref{optimal_problem} can be written as
\begin{align} 
v & =  e^{ - \eta  x_0 e^{rT}} \!\!\!\inf_{\mathbf u \in \mathcal{U}}\mathbb{E}^{\P^{\mathbf u} } \left[ e^{\eta (\int_{0}^{T} \int_{0}^{+\infty } \gamma^\b(u_s^\b)e^{r(T-s)}  z m(\ud s,\ud x) +\int_{0}^{T} e^{r(T-s)} (c_1(u_s^\a) + c_2(u_s^\b) ) \ud s)} \right] \nonumber \\
&= e^{ - \eta  x_0 e^{rT}} \inf_{\mathbf u \in \mathcal{U}}\mathbb{E}^{\P^{\mathbf u} } \left[ e^{- \eta Y_T^{\mathbf u}} \right],\label{(P)}
\end{align}
where $\{Y_t^{\mathbf u},\ t \in [0,T]\}$ denotes the auxiliary wealth 
\begin{equation}\label{wealth_Y}
Y^{\mathbf u}_t = -\! \int_0^t \!\!\int_{0}^{+\infty }\gamma^\b(u_s^\b) e^{r(t-s)} z m(\ud s,\ud z) - \int_0^t e^{r(t-s)} (c_1(u_s^\a) + c_2(u_s^\b) ) \ud s.
\end{equation}

We give a dynamic formulation of the stochastic control problem in \eqref{(P)}, and we introduce, for any $\mathbf u\in\mathcal{U}$, the associated Snell-Envelope:
\begin{equation}
\label{eqn:W}
W^{\mathbf u}_t = \essinf_{\bar{u}\in\mathcal{U}(t,\mathbf u)}
{\mathbb{E}^{\P^{\bar{u}}}\biggl[e ^{-\eta Y^{\mathbf u}_T } \Big{|}\mathcal{F}_t\biggr]},\; \forall t \in [0,T],
\end{equation}
with $\mathcal{U}(t, {\mathbf u})$ defined, for an arbitrary effort  ${\mathbf u}\in\mathcal{U}$, as the restricted class of controls almost surely equal to $\mathbf u$ over $[0,t]$, i.e.
\begin{equation*}
\mathcal U(t, {\mathbf u}):=\Big\{ \bar {\mathbf u} \in \mathcal U: \bar {\mathbf u}_s = \mathbf u_s \ \text{a.s.} \ \text{for all} \ s\le t \Big\}.
\end{equation*}

\begin{remark}
    The Snell-Envelope is well defined in terms of $\P^0$-infimum because all the probability measures $\P^{\mathbf u}$ are equivalent to $\P^0$; moreover for $\bar {\mathbf u}, \bar {\mathbf v} \in \mathcal U(t,\mathbf u)$ then  $\P^{\bar {\mathbf u}}(A) = \P^{\bar{\mathbf v}}(A) $  for any  $A \in \mathcal{F}_t$.  See \cite{ek81} for further details.
\end{remark}

Let  $V=\{V_t,\ t \in [0,T]\}$ be the value process defined as follows
\begin{equation}
\label{eqn:V}
V_t = \essinf_{\bar{\mathbf u}\in\mathcal{U}_t}{\mathbb{E}^{\P^{ \bar {\mathbf u} }}\biggl[e^{-\eta e^{rT}(\overline{Y}^{\bar{\mathbf u}}_T-\overline{Y}^{\bar{\mathbf u}}_t)}\Big{|}  \mathcal{F}_t\biggr]},\ \forall t \in [0,T],
\end{equation}
where $\overline{Y}^{\mathbf u}_t = e^{-rt} Y^{\mathbf u}_t$, for any $t\in[0,T]$, with $Y^{\mathbf u}$ given in \eqref{wealth_Y}. 

\begin{remark}\label{remark:Kall}
    For any $\bar{\mathbf u} \in \mathcal U(t,\mathbf u)$ we have that
    $$
    {\mathbb{E}^{\P^{\bar {\mathbf u} }}
    \biggl[e^{-\eta e^{rT}(\overline{Y}^{\bar{ \mathbf u}}_T-\overline{Y}^{\bar{\mathbf u}}_t)} \Big{|} \mathcal{F}_t\biggr]}
    $$
    does not depend on $\mathbf u$. 
    Indeed, by the Kallianpur-Striebel formula we get that
    \begin{align*}
    &{\mathbb{E}^{\P^{\bar {\mathbf u} }}
    \biggl[e^{-\eta e^{rT}(\overline{Y}^{\bar{ \mathbf u}}_T-\overline{Y}^{\bar{\mathbf u}}_t)} \Big{|} \mathcal{F}_t\biggr]} \\
    & \qquad = \mathbb{E}^{\P^0 }\left[ \frac{L^{\bar {\mathbf u}}_T}{L^{\bar {\mathbf u}}_t} e^{\eta e^{rT}\left(  \int_t^T e^{-rs}\left(c_1(\bar u^\a_s) + c_2(\bar u^\b_s)\right) \,\ud s
+\int_t^T\int_0^{+\infty} e^{-rs} \bar u^\b_s z \,m(\ud s,\ud z)\right) } \Big{|}  \mathcal{F}_t \right],
    \end{align*}
    under $\P^0$, $m(\ud t, \ud z)$ has non-controlled compensator $\lambda^0_t F(\ud z) \ud t$, and   
    $$
    \frac{L^{\bar {\mathbf u}}_T}{L^{\bar {\mathbf u}}_t} = \exp { \left( - \int_{t}^{T} (\gamma^\a(\bar u_s^\a)- 1) \lambda_s^0\ud s + \int_{t}^{T} \int_0^{+\infty}\log{\gamma^\a(\bar u_s^\a)} m(\ud s,\ud z) \right) }
    $$
    only depends on $\bar u_s$, for $t \le s \le T$.
    \end{remark}
Thanks to Remark \ref{remark:Kall}, we have that for every $\mathbf u \in \mathcal{U}$ the essential infimum over $\mathcal U_t$ in \eqref{eqn:V} coincides with that under $\mathcal U(t, {\mathbf u})$ and so 
\begin{equation}\label{eqn:W^{u,e}_and_V}
W^{\mathbf u}_t = e^{-\eta \overline{Y}^{\mathbf u}_t e^{rT}} V_t %=  e^{\eta (u\int_{0}^{t} \int_{0}^{\infty } e^{r(T-s)} z m^{(1)}(ds,dz) + \int_{0}^{t} e^{r(t-s)} c(e_s ) ds)} V_t,
\end{equation}
and, in turn, choosing the null effort $\mathbf u =(u^\a_t , u^\b_t)=(0,0)$, we find:
\begin{equation}
\label{eqn:VisWI}
V_t = e^{\eta \overline{Y}^{\mathbf 0}_t e^{rT}} W^{\mathbf 0}_t,\quad  \forall t \in [0,T],
\end{equation}
where $W^{\mathbf 0}=\{ W_t^{\mathbf 0},\; t \in [0,T]\}$ is the Snell Envelope associated to null effort, that is
\begin{equation*}
\label{eqn:W0}
W^{\mathbf 0}_t = \essinf_{\bar{\mathbf u}\in\mathcal{U}(t,\mathbf 0)}
{\mathbb{E}^{\P^{\bar{\mathbf u}} }\biggl[e ^{-\eta Y^{\bar{\mathbf u}}_T } \Big{|}\mathcal{F}_t\biggr]}, \quad \forall t \in [0,T].
\end{equation*}
Our aim now is working on a BSDE characterization of $W^{\mathbf 0}$, which also gives a full description of the value process $V$ and of an optimal effort $\mathbf u^*=\{ (u^{\a,*}_t, u^{\b,*}_t), \ t \in [0,T]\}$.
\begin{definition}
We define the following classes of stochastic processes:
\begin{itemize}
\item $\mathcal{L}^\b$ denotes the space of c\`adl\`ag $\mathbb{F}$-adapted processes $R=\{R_t, \ t \in [0,T]\}$ such that:
\[
\mathbb{E}^{\P^0}\left[\int_0^T |R_t|^2 \ud t\right] <+\infty.
\]
\item $\widehat{\mathcal{L}}$ denotes the space of $[0,+\infty)$-indexed $\mathbb{F}$-predictable random fields $\Theta=\{ \Theta(t,z),\ t \in [0,T],\ z \in [0, + \infty)\}$ such that:
\begin{equation}
\mathbb{E}^{\P^0}\left[\int_0^T\int_0^{+\infty} |\Theta(t,z)|^2 \lambda_t^0 F(\ud z) \ud t \right]  <+\infty.
\end{equation}

\item ${\mathcal{L}}^{\perp}$ denotes the space of $(\bF, \P^0)$-square integrable martingales $M=\{M_t,\ t \in [0,T]\}$ orthogonal to the compensated random measure $\widetilde m^0(\ud t, \ud z)$.  This means that for any $\Theta \in \widehat{\mathcal{L}}$, $M$ is orthogonal to the martingale $\{ \int_0^t\int_0^{+\infty} \Theta(t,z) \widetilde m^0(\ud s,\ud z), t \in [0,T]\}.$

\end{itemize}
\end{definition}

Now, we provide a general verification result, which will be useful in the sequel.
%Section \ref{bsde approach}. \\

\begin{proposition}\label{VT1}
Suppose there exists an $\mathbb{F}$-adapted process $D=\{D_t,\ t \in [0,T]\}$ such that:
\begin{itemize}
\item[(i)] $\{ D_t e^{ - \eta \overline{Y}^{\mathbf u}_t e^{rT}},\ t \in [0, T]\}$ is an $(\mathbb{F},\P^{\mathbf u})$-submartingale for any $\mathbf u \in \mathcal{U}$ and an $(\mathbb{F},\P^{\mathbf u^*})$-martingale for some $\mathbf u^* \in \mathcal{U}$;
\item[(ii)] $D_T = 1$ $\P^0-a.s.$
\end{itemize}
Then, $D_t = V_t$ $\P^0$-a.s. for each $t \in [0,T]$ and $\mathbf u^* $ is an optimal effort. 
Moreover, $W^{\mathbf u}_t = D_t e^{ - \eta \overline{Y}^{\mathbf u}_t e^{rT}}$ is an $(\mathbb{F},\P^{\mathbf u})$-submartingale for any $\mathbf u \in \mathcal{U}$ and $W^{\mathbf u^*}_t$ is an $(\mathbb{F},\P^{\mathbf u^*})$-martingale.
\end{proposition}

\begin{proof} 
In view of the terminal condition and the submartingale property, for every $\mathbf u \in \mathcal{U}$ and $t \in [0,T]$ we have
\begin{equation*}
   \mathbb{E}^{\P^{\mathbf u} }\biggl[ D_T e^{ - \eta \overline{Y}^{\mathbf u}_T e^{rT}} \mid \mathcal{F}_t \biggr ] \ge D_t e^{ - \eta \overline{Y}^{\mathbf u}_t  e^{rT}}, \quad \P^0-{\rm a.s.}
\end{equation*}
so that
\begin{equation*}
    D_t \leq  \mathbb{E}^{\P^{\mathbf u} }\biggl[ e^{ - \eta (\overline{Y}^{\mathbf u}_T -  \overline{Y}^{\mathbf u}_t) e^{rT}}\mid \mathcal{F}_t \biggr ], \quad \P^0-{\rm a.s.}
\end{equation*}
which implies $D_t \leq V_t$ $\P^0$-a.s. for each $t \in [0,T]$. On the other hand, for $\mathbf u^* \in \mathcal{U}$, we have that
\begin{equation*}
D_t=\mathbb{E}^{\P^{\mathbf u^*} }\biggl[ e^{ - \eta (\overline{Y}^{\mathbf u^*}_T -  \overline{Y}^{\mathbf u^*}_t) e^{rT}}  \mid \mathcal{F}_t \biggr ]  \ge V_t, \quad \P^0-{\rm a.s.},
\end{equation*}
implying that $D_t=V_t$ $\P^0$-a.s. for each $t \in [0,T]$. From \eqref{eqn:W^{u,e}_and_V} the last part of the statement follows immediately.
\end{proof}

Notice that the last part of the statement of Proposition \ref{VT1} corresponds to the well-known Bellman optimality principle (see e.g. \cite{ek81,lim-quenez_2011}). \\
We now provide a verification result in terms of a suitable BSDE. In \cite{BCCS_2025}, a similar methodology was employed to address the self-protection problem within a contagion-risk model.

\begin{theorem}[Verification Theorem]\label{VerThm}
Assume that the second condition in \eqref{Ass2} holds and
\begin{equation}\label{lambda_cond}
    \mathbb E^{\P^0} \left[ e^{4 \int_0^T \lambda_s^0 ds} \right] < +\infty, \quad \mathbb E^{\P^0} \left[ \int_0^T    (\lambda_t^0)^2 \ud t \right]< + \infty.
\end{equation} 
Let $(R,\Theta^{R}, M) \in \mathcal{L}^\b \times \widehat{\mathcal{L}} \times \mathcal{L}^\perp$ be a solution to the BSDE under $\P^0$ 
\begin{equation} \label{bsde}
\begin{split}
R_t & = \Xi - \int_t^T \int_0^{+ \infty} \Theta^{R} (s, z) \widetilde m^{0}(\ud s, \ud z) +\\
    & \qquad - \int_t^T \esssup_{\mathbf u \in \mathcal U} f ( s, R_{s^-} , \Theta^{R}(s,\cdot),u_s^\a,u_s^\b)\,\ud s + M_T-M_t,
 \end{split}
 \end{equation}
with terminal condition 
$$\Xi = e^{- \eta Y^{\mathbf 0}_T} = e^{\eta e^{rT} \int_0^T\int_0^{+\infty} e^{-rt} z m(\ud t, \ud z)},
$$ 
where
\begin{equation} 
\begin{split}
&f ( t, R_{t^-} , \Theta^{R}(t, \cdot ) , u_t^\a, u_t^\b) )=
- R_{t^-} \eta e^{r(T-t)} (c_1(u_t^\a) + c_2(u_t^\b)) + \nonumber \\ 
&+ \int_0^{+\infty} \Theta^R(t,z) \lambda^0_{t^-} F(\ud z) + \nonumber \\
& -  \gamma^\a(u_t^\a) \lambda^0_{t^-} \int_0^{+\infty} \!\![(\Theta^R(t,z) + R_{t^-}) ( e^{-\eta e^{r(T-t)}z (1-\gamma^\b(u_t^\b))} -1) + \Theta^R(t,z)] F(\ud z) \label{ftilde}
\end{split}
    \end{equation}

\noindent Then, $R_t = W^{\mathbf 0}_t$ $\P^0$-a.s. for each $t\in [0,T]$ and any process $\mathbf u^*=(u^{\a,*},u^{\b,*}) \in \mathcal{U}$ which satisfies for any  $t\in[0,T]$:
\begin{equation*}\label{eqn:u*esssup}
f( t, R_{t^-} , \Theta^{R}(t, \cdot ) , u^{\a,*}_t,u^{\b,*}_t) 
	= \esssup_{\mathbf u \in \mathcal{U}}  f( t, R_{t^-} , \Theta^{R}(t, \cdot ) , u_t^\a,u_t^\b),\ \P^0-{\rm a.s.}
\end{equation*}
provides an optimal effort.
\end{theorem}
\begin{proof}
Let $(R,\Theta^{R}, M) \in \mathcal{L}^\b \times \widehat{\mathcal{L}} \times \mathcal{L}^\perp $ be a solution to the BSDE \eqref{bsde}.
Define the process $D=\{D_t,\ t \in [0,T]\}$ by setting $D_t := R_t e ^{\eta {\overline Y}^{\mathbf 0}_t e^{rT}}$, $t \in [0,T]$. We will prove that $D$ satisfies conditions $(i)$ and $(ii)$ in Proposition \ref{VT1} hence $D_t = V_t$ $\P^0$-a.s. and so $R_t=W^{\mathbf 0}_t$ $\P^0$-a.s., for every $t \in [0,T]$.
First, observe that $(ii)$ is trivially satisfied because $D_T := R_T e ^{\eta \overline Y^{u,0}_T e^{rT}} = 1$, $\P^0$-a.s., in view of the terminal condition of the BSDE \eqref{bsde}. 
Next, we focus on item $(i)$. By It\^o's product rule we have that
\begin{align*}\label{ito}
&  \ud (D_t \ e^{ - \eta \overline{Y}^{\mathbf u}_t e^{rT}})  =   \ud  ( e^{  \eta (\overline{Y}^{\mathbf 0}_{t} - \overline{Y}^{\mathbf u}_{t})e^{rT}} \ R_t) \\
& =  e^{\eta (\overline{Y}^{\mathbf 0}_{t^-}-\overline{Y}^{\mathbf u}_{t^-}) e^{rT}}\ud R_t + R_{t-} \ud ( e^{  \eta (\overline{Y}^{\mathbf 0}_t - \overline{Y}^{\mathbf u}_t)e^{rT}} ) + \ud \left ( \sum_{s \leq t} \Delta R_s \ \Delta \big(e^{  \eta (\overline{Y}^{\mathbf 0}_s - \overline{Y}^{\mathbf u}_s)e^{rT}}\big) \right).
\end{align*}
Observing that
\begin{equation}\label{nuova}
\overline{Y}^{\mathbf 0}_t -   \overline{Y}^{\mathbf u}_t=  - \int_0^t \int_{0}^{+\infty }(1-\gamma^\b(u_s^\b)) e^{-rs} z m(\ud s,\ud z) + \int_0^t e^{-rs} (c_1(u_s^\a) + c_2(u_s^\b) ) \ud s,
\end{equation}
by It\^o's formula we get
\begin{align*}\ud ( e^{  \eta (\overline{Y}^{\mathbf 0}_t - \overline{Y}^{\mathbf u}_t)e^{rT}} ) = & \eta e^{rT} e^{  \eta (\overline{Y}^{\mathbf 0}_t - \overline{Y}^{\mathbf u}_t)e^{rT}} e^{-rt} [c_1(u_t^\a) + c_2(u_t^\b)] \ud t\\
&+ e^{  \eta (\overline{Y}^{\mathbf 0}_{t^-} - \overline{Y}^{\mathbf u}_{t^-})e^{rT}} \int_0^{+\infty} (e^{ - \eta (1-\gamma^\b(u_t^\b)) z e^{r(T-t)}} -1) m(\ud t, \ud z)
\end{align*}
and 
\begin{equation*}
\begin{split}
&\ud \left ( \sum_{s \leq t} \Delta R_s \ \Delta \big(e^{  \eta (\overline{Y}^{\mathbf 0}_s - \overline{Y}^{\mathbf u}_s)e^{rT}}\big) \right) \\
& \qquad = \int_0^{+\infty} \Theta^{R} (t, z) e^{  \eta (\overline{Y}^{\mathbf 0}_{t^-} - \overline{Y}^{\mathbf u}_{t^-})e^{rT}} (e^{ - \eta (1-\gamma^\b(u_t^\b)) z e^{r(T-t)}} -1) m(\ud t, \ud z).
\end{split}
\end{equation*}
Finally, recalling by \eqref{bsde} that 
$$
\ud R_t =   \int_0^{+ \infty} \Theta^{R} (t, z) \widetilde m^{0}(\ud t, \ud z) + \ud M_t + \esssup_{\mathbf u \in \mathcal U} f ( t, R_{t^-} , \Theta^{R}(t,\cdot),\mathbf u_t)\,\ud t ,
$$
and 
$$
\widetilde m^{0}(\ud t,\ud z) = \widetilde m^{\mathbf u}(\ud t,\ud z)  + \nu^{\mathbf u}(\ud t, \ud z)- \nu^0(\ud t, \ud z), 
$$
in view of \eqref{comp} and \eqref{ftilde}, we find:
\begin{align} 
& \ud (D_t \ e^{ - \eta \overline{Y}^{\mathbf u}_t e^{rT}})  \\
& = \ud \widetilde M^{\mathbf u}_t +  e^{  \eta (\overline{Y}^{\mathbf 0}_t - \overline{Y}^{\mathbf u}_t)e^{rT}} \left( \esssup_{\bar{\mathbf u} \in \mathcal U} f ( t, R_{t^-} , \Theta^{R}(t, \cdot),\bar{\mathbf u}_t) -  f ( t, R_{t^-} , \Theta^{R}(t, \cdot),\mathbf u_t) ) \right)\ud t, \label{eq:verif}
\end{align}
where $\widetilde M^{\mathbf u}=\{\widetilde M^{\mathbf u}_t,\ t \in [0,T]\}$ is an $(\mathbb F,\P^{\mathbf u})$-local martingale given by 
\begin{equation}\label{Mu}
\begin{split}
& \ud \widetilde M_t^{\mathbf u} = e^{  \eta (\overline{Y}^{\mathbf 0}_{t-} - \overline{Y}^{\mathbf u}_{t-})e^{rT}} \left(\int_0^{+ \infty} \Theta^{R} (t, z) \widetilde m^{\mathbf u}(\ud t, \ud z) + \ud M_t\right) \\ 
& \quad +
e^{  \eta (\overline{Y}^{\mathbf 0}_{t-} - \overline{Y}^{\mathbf u}_{t-})e^{rT}} \int_0^{+\infty}(R_{t^-} + \Theta^R(t,z)) (e^{ - \eta (1-\gamma^\b(u_t^\b)) z e^{r(T-t)}} -1) \widetilde m^{\mathbf u}(\ud t, \ud z). 
\end{split}
\end{equation}
We now prove that $\widetilde M^{\mathbf u}$ is a true $(\mathbb F,\P^{\mathbf u})$-martingale. 
We first show that
$$
\mathbb E^{\P^\mathbf u } \left[ \int_0^T \\\int_0^{+\infty}\!\! e^{  \eta (\overline{Y}^{\mathbf 0}_{t-} - \overline{Y}^{\mathbf u}_{t-})e^{rT}}\!\!\left( | \Theta^{R}(t,z) |  + |R_{t^-} + \Theta^{R}(t,z)||e^{ - \eta (1-\gamma^\b(u_t^\b)) z e^{r(T-t)}} -1 | \right)\nu^{\mathbf u} (\ud t, \ud z)  \right] < +\infty.
$$
Boundedness of the cost functions $c_i(u^{(i)})$, $i=1,2$,  and $0<\gamma^\b(u^\b)\leq 1$, imply boundedness of the process
 $e^{ \eta (\overline{Y}^{\mathbf 0}_{t-} - \overline{Y}^{\mathbf u}_{t-})e^{rT}}$.
So, in view of \eqref{comp},  we get
\begin{align*}
     &\mathbb E^{\P^\mathbf u } \left[ \int_0^T \!\!\int_0^{+\infty} e^{  \eta (\overline{Y}^{\mathbf 0}_{t-} - \overline{Y}^{\mathbf u}_{t-})e^{rT}} | \Theta^{R}(t,z) | \nu^{\mathbf u} (\ud t, \ud z)  \right]\\
     & \qquad \qquad \qquad \leq C_1 \mathbb  E^{\P^0} \left[ L_T^{\mathbf u} \int_0^T \int_0^{+\infty} | \Theta^{R}(t,z) |  \nu^0 (\ud t, \ud z) \right] & \\
    &  \qquad \qquad \qquad = C_1 \mathbb E^{\P^0} \left[ \int_0^T   \int_0^{+\infty} L_T^{\mathbf u}  | \Theta^{R}(t,z) | \lambda_t^0 F(\ud z) \ud t  \right], 
\end{align*}
for a suitable constant $C_1>0$.
By using the inequality $ab \le \frac12 (a^2 + b^2)$, for all $a,b \in \R$, with the choice $a=L_T^{\mathbf u}$ and $b=\Theta^{R}(t,z)$, we find:
\begin{equation}\label{eq:Ltheta}
\begin{split}
& \mathbb E^{\P^0} \left[ \int_0^T   \int_0^{+\infty} L_T^{\mathbf u}  | \Theta^{R}(t,z) | \lambda_t^0 F(\ud z) \ud t  \right]  \\
& \; \le \frac12 \mathbb E^{\P^0} \left[ \int_0^T  \left( L_T^{\mathbf u} \right)^2 \lambda_t^0 \ud t  \right] + \frac12 \mathbb E^{\P^0} \left[ \int_0^T   \int_0^{+\infty} | \Theta^{R}(t,z) |^2 \lambda_t^0 F(\ud z) \ud t  \right]. 
\end{split}
\end{equation}
The second expectation on the right-hand side in \eqref{eq:Ltheta} is finite since $\Theta^{R}$ belongs to $\widehat{\mathcal{L}}$.
For what concerns the first term, applying again the inequality $ab \le \frac12 (a^2 + b^2)$, for all $a,b \in \R$, we get, via  \eqref{stima}:
\begin{eqnarray*}
\mathbb E^{\P^0} \left[ \int_0^T  \left( L_T^{\mathbf u} \right)^2 \lambda_t^0 \ud t  \right]& \le & \frac12 \left( \mathbb E^{\P^0} \left[ \left( L_T^{\mathbf u} \right)^4  \right]\cdot T +
\mathbb E^{\P^0} \left[ \int_0^T    (\lambda_t^0)^2 \ud t \right] \right) \\
& \le & \frac12 \left( \mathbb E^{\P^0} \left[ e^{4 \int_0^T \lambda_s^0 \ud s}\right]\cdot T +
\mathbb E^{\P^0} \left[ \int_0^T    (\lambda_t^0)^2 \ud t \right] \right),
\end{eqnarray*}
which is finite thanks to assumption \eqref{lambda_cond}.
We now observe that 
\begin{align*}
&\mathbb E^{\P^\mathbf u } \left[ \int_0^T \int_0^{+\infty} |R_{t^-} + \Theta^{R}(t,z)||e^{ - \eta (1-\gamma^\b(u_t^\b)) z e^{r(T-t)}} -1 | \nu^{\mathbf u} (\ud t, \ud z)  \right] \\\
& \leq 
\mathbb E^{\P^\mathbf u } \left[ \int_0^T \int_0^{+\infty} \left(|R_{t^-}| + |\Theta^{R}(t,z)|\right)\lambda_t^0 F(\ud z) \ud t \right] \\
&= \mathbb E^{\P^0 } \left[ L^{\mathbf u}_T \int_0^T \int_0^{+\infty} \left(|R_{t^-}| + |\Theta^{R}(t,z)|\right)\lambda_t^0 F(\ud z) \ud t \right], 
\end{align*}
which is finite if we prove that $\mathbb E^{\P^0 } \left[ L^{\mathbf u}_T \int_0^T \int_0^{+\infty} |R_{t^-}| \lambda_t^0 F(\ud z) \ud t \right]$ is finite. So
\begin{eqnarray*}
\mathbb E^{\P^0 } \left[ L^{\mathbf u}_T \int_0^T \int_0^{+\infty} |R_{t^-}| \lambda_t^0 F(\ud z) \ud t \right] = \mathbb E^{\P^0 } \left[  \int_0^T L^{\mathbf u}_T |R_{t^-}| \lambda_t^0 \ud t \right]\\
\leq \frac12 \mathbb E^{\P^0} \left[ \int_0^T  \left( L_T^{\mathbf u} \right)^2 \lambda_t^0 \ud t  \right] + \frac12 \mathbb E^{\P^0} \left[ \int_0^T    | R_t |^2 \lambda_t^0 \ud t  \right] < +\infty,
\end{eqnarray*}
since
%because $R \in \mathcal{L}^\b$ and 
$$\mathbb E^{\P^0} \left[ \int_0^T    | R_t |^2 \lambda_t^0 \ud t  \right] \leq \frac12 \mathbb E^{\P^0} \left[\int_0^T    | R_t |^2  \ud t  \right] + \frac12 \mathbb E^{\P^0} \left[ \int_0^T    (\lambda_t^0)^2 \ud t \right] < +\infty $$
$$\mathbb E^{\P^0} \left[ \int_0^T  \left( L_T^{\mathbf u} \right)^2 \lambda_t^0 \ud t  \right]  \leq \frac12 \mathbb E^{\P^0} \left[ e^{4 \int_0^T \lambda_s^0 \ud s}\right]\cdot T +  
\frac12 \mathbb E^{\P^0} \left[ \int_0^T    (\lambda_t^0)^2 \ud t \right] < +\infty,
$$
in view of Assumptions \ref{lambda_cond} and  $R \in \mathcal{L}^\b$.
To prove that for any $u \in {\mathcal U}$, $\widetilde M^{\mathbf u}$ is a true $(\mathbb F,\P^{\mathbf u})$-martingale, it remains to show that the term in \eqref{Mu} driven by $M$ is an $(\mathbb F,\P^{\mathbf u})$-martingale. We first note that the process 
 $e^{ \eta (\overline{Y}^{\mathbf 0}_{t-} - \overline{Y}^{\mathbf u}_{t-})e^{rT}}$ is bounded. It is therefore sufficient to show  that $M$ is an $(\bF,\P^{\mathbf u})$-martingale.
 This is equivalent to proving that $ML^{\mathbf u}$ is an $(\bF,\P^0)$-martingale, which follows from the fact that $M$ and $L^{\mathbf u}$ are $(\bF,\P^0)$-orthogonal martingales. 
 Indeed, in view of \cite[Proposition 4.15]{js2013} it is sufficient to observe that
\begin{align*}
    \langle M,L^{\mathbf u}\rangle_t & = \left\langle M,\int_0^\cdot L^{\mathbf u}_{s^-}\int_0^{+\infty}(\gamma^\a(u^\a_s) -1)\widetilde m^0(\ud s, \ud z) \right\rangle_t\\
    &=\int_0^t L^{\mathbf u}_{s^-} \left \langle \ud M_{s^-},\int_0^{+\infty}(\gamma^\a(u^\a_s) -1)\widetilde m^0(\ud s, \ud z)\right\rangle = 0, \quad t \in [0,T].
\end{align*}
Hence, from Equation \eqref{eq:verif} we can see that the process $\{ D_t e^{ - \eta \overline{Y}^{u,e}_t e^{rT}},\ t \in [0, T] \}$ is an $(\mathbb{F},\P^{\mathbf u})$-submartingale.

Moreover, we can observe that the process $ f( t, R_{t^-} , \Theta^{R}(t, \cdot ) , \mathbf u)$ is $\mathbb F$-predictable process for any $\mathbf u=(u^\a,u^\b)\in[0,\zeta_1]\times [0,\zeta_2]$ and continuous in  $\mathbf u$.
%({\color{red}{CHECK}}: \blublu{VERIFICATO: since \( f \) is continuous with respect to \( e^\a \) for each fixed value of \( e^\b \), and continuous with respect to \( e^\b \) for each fixed value of \( e^\a \), and  the domain is compact, then $f$ is continuous with respect to $e=(e^\a,e^\b)\in[0,\zeta_1]\times [0,\zeta_2]$
Consequently, from measurability selection results (see \cite{benevs1971existence}) and Proposition \ref{Adm}, there exists a maximizer ${\mathbf u}^*\in\mathcal{U}$. 
Finally, by choosing $\mathbf u={\mathbf u}^*$ in Equation \eqref{eq:verif}, it follows that the process $\{ D_t e^{ - \eta \overline{Y}^{u,e^*}_t e^{rT}},\ t \in [0, T]\}$ is an $(\mathbb{F},\P^{{\mathbf u}^*})$-martingale.
\end{proof}

We now discuss existence of the solution to the BSDE \eqref{bsde}. Due to unboundedness of the claim arrival intensity, the generator of the BSDE is not Lipschitz and so we rely on  \cite{PPS2018}, where existence results for BSDEs with stochastic Lipschitz generators are provided. %The authors in {\color{red} \cite{PPS2018} extended existence and uniqueness results for BSDEs driven by a general càdlàg martingale and with stochastic Lipschitz generators studied in the seminal paper \cite{El Karoui_Huang}.} 
Precisely, to apply \cite[Theorem 3.5]{PPS2018}, we need to strengthen the assumptions in Theorem \ref{VerThm}.

\begin{theorem} 
\label{thm:bsde} 
Let us assume  \begin{equation}
    \label{GEN}
\forall a>0, \quad \mathbb{E}^{\P^0} \left[ e^{a \int_0^T \lambda_s^0 ds} \right] < +\infty, \quad \mathbb{E}^{\P^0}\left[ e^{aJ_T}\right]<+\infty.
\end{equation}
Then, there exists a solution in $\mathcal{L}^\b \times \widehat{\mathcal{L}} \times \mathcal{L}^{\perp}$ to the BSDE \eqref{bsde}.
\end{theorem}

The proof is postponed to Appendix \ref{app:proof}. 

 \begin{remark}\label{REWEAK}
The conditions in \eqref{GEN} can be relaxed to the following integrability requirements:
\begin{equation}\label{Weak}\mathbb E^{\P^0} \left[e^{4\eta e^{rT} J_T} \right]< + \infty, \quad \mathbb E^{\P^0} \left[e^{ 8\widehat \beta \int_0^T  \lambda^0_{s}\ud s }\right] < + \infty
\end{equation}
for a given constant $\widehat{\beta} > 0$. Specifically, the constant $\widehat{\beta}$ is defined such that condition \eqref{Mphi} holds with $\Phi < \frac{1}{18e }$; see Appendix~\ref{app:proof} for further details.
\end{remark}

Combining Theorems \ref{VerThm} and \ref{thm:bsde} yields our main result, characterizing the Snell envelope 
$W^0$ and the optimal prevention strategies and, as a consequence, the value process $V$, thanks to equation \eqref{eqn:VisWI}.
\begin{theorem}\label{MainResult}
Let us assume conditions in \eqref{GEN}. 
Then, $(W^{\mathbf 0}, \Theta, M) \in \mathcal{L}^\b \times \widehat{\mathcal{L}} \times \mathcal{L}^{\perp}$ is the unique solution to the following BSDE  
\begin{equation} \label{bsde1}
\begin{split}
W^{\mathbf 0}_t & = \Xi - \int_t^T \int_0^{+ \infty} \Theta(s, z) \widetilde m^{0}(\ud s, \ud z) + \\
	& - \int_t^T \esssup_{\mathbf u \in \mathcal U} f ( s, W^{\mathbf 0}_{s^-} , \Theta(s,\cdot),u_s^\a,u_s^\b)\,\ud s   + M_T-M_t,
    \end{split}
 \end{equation}
with terminal condition 
$$\Xi = e^{- \eta Y^{\mathbf 0}_T} = e^{\eta e^{rT} \int_0^T\int_0^{+\infty} e^{-rt} z m(\ud t, \ud z)}$$ and where
%\begin{equation}
\begin{align*}\label{ftilde1}
&f( t, W^{\mathbf 0}_{t^-} , \Theta (t, \cdot ) , u_t^\a, u_t^\b)  = 
- W^{\mathbf 0}_{t^-} \eta e^{r(T-t)} (c_1(u_t^\a) + c_2(u_t^\b)) + \\ &+  \int_0^{+\infty} \Theta(t,z) \lambda^0_{t^-} F(\ud z)\\
& -  \gamma^\a(u_t^\a) \lambda^0_{t^-} \int_0^{+\infty} [(\Theta(t,z) + W^{\mathbf 0}_{t^-}) ( e^{-\eta e^{r(T-t)}z (1-\gamma^\b(u_t^\b))} -1) + \Theta(t,z)] F(\ud z).
    \end{align*}
    Moreover, any process $\mathbf u^*=(u^{\a,*},u^{\b,*})$ which realizes the essential supremum in \eqref{bsde1} provides an (admissible) optimal prevention strategy.
%\end{equation}

\end{theorem}
 \begin{remark}\label{REweak1}
According to Remark~\ref{REWEAK}, Theorem \ref{MainResult} remains valid if the original conditions in  \eqref{GEN} are replaced by the weaker assumptions \eqref{lambda_cond} and \eqref{Weak}.
\end{remark}

We emphasize that our results are valid without any assumptions on the structure of the filtration $\bF$. Therefore, they are applicable to a wide range of risk models. When more information on filtration $\bF$ is available and a martingale representation theorem can be applied, it is possible to specify the orthogonal martingale $\{M_t,\ t\in [0,T]\}$. 

In the following, we give some examples of risk models for which Theorem \ref{MainResult} applies.
\subsection{Some classes of risk models}\label{Exa}
 %Our BSDE approach can apply to a wide class of risk models. 
 \begin{itemize}
 \item[(i)]  We consider the case where $\lambda^0$ is a bounded process, that is $\lambda^0_t \leq \Lambda$, $\mathbb{P}^0-a.s.$ for any $t\in [0,T]$, with $\Lambda >0$ a constant.  The claim arrival intensity $N=\{N_t,\ t \in [0,T]\}$ can be defined as a  conditional Poisson process having $\lambda^0$ as intensity, that is
 $$
 \P^0(N_t-N_s=k| \F_T^{\lambda^0} \vee \F_s) = \frac{(\int_s^t \lambda^0_v \ud v)^k}{k!} e^{-\int_s^t \lambda^0_v \ud v}, \ 0\leq s\leq t\leq T,\ k=0,1, \ldots
 $$
 Assuming $N$ and the sequence of claim sizes $\{Z_i\}_{i\geq 1}$ to be independent and, $\mathbb{E}^{\P^{\mathbf 0}}[e^{a Z }] < + \infty$ for $a=4\eta e^{rT}$, we have that  conditions \eqref{lambda_cond} and \eqref{Weak} are satisfied. Indeed, 
 \begin{align*}
     \mathbb{E}^{\P^{\mathbf 0}}[e^{aJ_T }] & = \mathbb{E}^{\P^{\mathbf 0}}[ \mathbb{E}^{\P^{\mathbf 0}}[e^{a J_T }|N_T]] = \mathbb{E}^{\P^{\mathbf 0}}[ (\mathbb{E}^{\P^{\mathbf 0}}[e^{a Z}])^{N_T}] \\
& = \mathbb{E}^{\P^{\mathbf 0}}\left [ e^{( \mathbb{E}^{\P^{\mathbf 0}}[e^{a Z}] -1) \int_0^T \lambda^0_s \ud s}\right] \leq  e^{( \mathbb{E}^{\P^{\mathbf 0}}[e^{a Z}] -1)\Lambda T} .
\end{align*}
According to Remark \ref{REweak1}, therefore Theorem \ref{MainResult} applies. 

 This model encompasses several important cases, including: constant claim arrival intensity; Markov-modulated models where the intensity is driven by a finite-state continuous-time Markov chain independent of the claim sizes $\{Z_i\}_{i \geq 1}$; and stochastic factor models in which $\lambda^0_t = \Gamma(t, Y_{t^-})$, with $\Gamma(t, y)$ a bounded function and $\{Y_t,\ t \in [0,T]\}$ a càdlàg process independent of $\{Z_i\}_{i \geq 1}$.

 \item[(ii)] More complex models allow for unbounded claim arrival intensity and self-exciting features. For instance,  the contagion model proposed in \cite{BCCS_2024, BCCS_2025} combines Hawkes and Cox processes with shot noise intensity.  
Precisely, $\{\lambda^0_t,\ t \in [0,T]\}$ follows the dynamics:
\begin{equation}\label{intensity_int}
\lambda_t^0 = \beta + (\lambda_0^0 - \beta)  e^{-\alpha t} + \sum_{i=1}^{N_t} e^{-\alpha (t -  T_i)} \ell(Z_i) +  \sum_{i=1}^{\widetilde N_t}  e^{-\alpha (t - \widetilde T_i)} \widetilde Z_i,
\end{equation}
where $\lambda_0^0>0$, $\beta>0 $ and $\alpha>0$ play the role of initial datum, mean-reversion speed and long-term mean, $\ell$ is a positive function modeled self-exciting jump sizes, $\{\widetilde Z_i\}_{i \geq 1}$ is a sequence of nonnegative i.i.d.  random variables, $\widetilde N=\{\widetilde N_t = \sum_{i \geq 1}\mathbbm{1}_{\{ \widetilde T_i \leq t\} },\ t \in [0,T]\}$,  is a Poisson process with intensity $\rho>0$, $\{Z_i\}_{i \geq 1}$, $\{\widetilde Z_i\}_{i \geq 1}$  and $\widetilde N$ are assumed mutually independent.
In this model $\lambda^0$ is an unbounded process, and the claim arrival process is not independent of claim sizes.
Since the process given in \eqref{intensity_int} is càdlàg, we take as predictable intensity its left continuous version $\{\lambda^0_{t^-},\ t \in [0,T]\}$.  Under the assumptions, for any $a >0$
$$\mathbb{E}^{\P^{\mathbf 0}}[e^{a Z}] < +\infty,\quad \mathbb{E}^{\P^{\mathbf 0}}[e^{a \ell(Z) }] < +\infty, \quad \mathbb{E}^{\P^{\mathbf 0}}[e^{a \widetilde Z }]< +\infty$$
hypotheses \eqref{GEN} in Theorem \ref{MainResult} are satisfied (see \cite[Proposition 1]{BCCS_2025}  and \cite[Lemma 4.6, Lemma B.1.]{BCCS_2024} for details).
\end{itemize}

\section{The self-protection problem for an insurer and incentives}\label{sec:self-prot}

We now discuss the special case of pure self-protection, that is, the situation in which the agent can only take actions aimed at reducing the claim arrival intensity. We consider the same setting as in \cite{bhk2023,BCCS_2025}. Precisely, dynamic self-insurance is replaced now by the choice of a market insurance protection. We then show how our general framework recovers the results in \cite{bhk2023,BCCS_2025} as a special case.

Specifically, the agent can enter into a proportional insurance contract over the period $[0,T]$, selecting a retention level $\theta \in [0,1]$, and pay a premium $\Pi(\theta)\geq 0$ at time $t=0$. The insurance coverage is assumed to remain fixed throughout the interval $[0,T]$.
Moreover, to incentivize self-protection actions, the insurance company offers the policyholder a terminal reimbursement given by given by an $\F_T$-random variable, $\xi^\theta \in [0, \Pi(\theta)e^{r T}]$.
The optimization problem can now be formulated as follows:
 \begin{equation}\label{(P1)}
 \begin{split}
v^\a & =  \inf_{(\theta, u^\a) \in [0,1] \times \mathcal{U}^\a} \!e^{ \eta  e^{rT} (\Pi(\theta) -x_0)}
\mathbb{E}^{\P^{u^\a}}\!\left[ e^{\eta \left( \theta \int_{0}^{T} \int_{0}^{+\infty }  e^{r(T-s)}  z m(\ud s,\ud x) +\int_{0}^{T} e^{r(T-s)} c_1(u_s^\a) \ud s \right)} \right] \\
&= \inf_{(\theta, u^\a) \in [0,1] \times \mathcal{U}^\a} e^{  \eta  e^{rT} ( \Pi(\theta) -x_0)} \mathbb{E}^{\P^{u^\a}} \!\left[ e^{- \eta \left(Y_T^{\theta, u^\a} + \xi^\theta\right)} \right],
\end{split}
\end{equation}
where 
\begin{equation*}\label{wealth_Y1}
Y_t^{\theta, u^\a} = - \int_0^t \int_{0}^{+\infty } \theta e^{r(t-s)} z m(\ud s,\ud z) - \int_0^t e^{r(t-s)} c_1(u_s^\a) \ud s, \quad t\in [0,T],
\end{equation*}
and $\mathcal{U}^\a$ is the subset of admissible strategies $\mathbf u_t = (u^\a_t,0) \in \mathcal{U}$.
Problem \eqref{(P1)} reads as 
\begin{equation*} \label{(P2)} 
v^\a= \inf_{\theta \in [0,1]} e^{  \eta  e^{rT} ( \Pi(\theta) -x_0)} v_\theta\end{equation*}
where
\begin{align}\label{(P2bis)}
v_\theta =  \inf_{u^\a  \in \mathcal{U}^\a} 
\mathbb{E}^{\P^{u^\a} } \left[ e^{- \eta \left(Y_T^{\theta, u^\a} + \xi^\theta\right)} \right].
\end{align}
Fixed $\theta \in [0,1]$, we focus on the self-protection problem \eqref{(P2bis)}, and then the minimization over $\theta$ can be discussed as in \cite{BCCS_2025}.
We now introduce the Snell-Envelope and the value process associated to problem \eqref{(P2bis)}.
Precisely, the Snell-Envelope associated to null self-protection strategy, is
\begin{equation*}
\label{eqn:WC}
W^{\theta, 0}_t = \essinf_{\bar{ u}^\a\in\mathcal{U}^\a(t,0)}
{\mathbb{E}^{\P^{\bar{u}^\a} }\biggl[e ^{-\eta \left(Y^{\theta,\bar{u}^\a}_T + \xi ^\theta\right) } \mid \mathcal{F}_t\biggr]}, \quad \forall t \in [0,T].
\end{equation*}
In \cite{bhk2023} the optimization problem is discussed when the underlying claim arrival intensity is constant and in \cite{BCCS_2025}
for the contagion model (see Subsection \ref{Exa}(ii)). We now provide a result which extends \cite[Theorem 3 ]{BCCS_2025} (in the case $\gamma^\a=\gamma^\b$) to general risk models. 

\begin{theorem}\label{SELF} Let us assume \eqref{GEN}.
Then, $(W^{\theta, 0}, \Theta^\theta, M^\theta) \in \mathcal{L}^\b \times \widehat{\mathcal{L}}\times \mathcal{L}^{\perp}$ is the unique solution to the following BSDE  
\begin{equation} \label{bsdeC}
\begin{split}
W^{\theta, 0}_t & = \Xi^\theta - \int_t^T \int_0^{+ \infty} \Theta^\theta(s, z) \widetilde m^{0}(\ud s, \ud z) + \\
	&- \int_t^T \esssup_{u^\a \in \mathcal{U}^\a} f^\a ( s, W^{\theta, 0}_{s^-} , \Theta^\theta(s,\cdot),u_s^\a)\,\ud s+ M^\theta_T-M^\theta_t,
    \end{split}
 \end{equation}
with terminal condition 
$$\Xi^\theta = e^{- \eta  \left(Y^{\theta, 0}_T + \xi^\theta\right)}
$$ and where
%\begin{equation}
\begin{equation}\label{ftildeC}
\begin{split}
&f^\a ( t, W^{\theta, 0}_{t^-} , \Theta (t, \cdot ) , u_t^\a) \\
& \qquad = 
- W^{\theta, 0}_{t^-} \eta e^{r(T-t)} c_1(u_t^\a) + 
(1 -  \gamma^\a(u_t^\a)) \lambda^0_{t} \int_0^{+\infty}  \Theta(t,z) F(\ud z).
\end{split}
    \end{equation}
    Moreover, any process $u^{(1),*}$ which realizes the essential supremum in \eqref{bsdeC} provides an (admissible) optimal self-protection strategy for the optimization problem \eqref{(P2bis)} and $v_\theta=W^{\theta,0}_0.$
%\end{equation}
\end{theorem}

\begin{proof}
The case $\theta=1$ and $\xi^\theta=0$ follows directly from Theorem \ref{MainResult} taking $\gamma^\b(u^\b)=1$, for any $u^\b \in [0,\zeta_2]$. 
The proof in the general case ($\theta \in [0,1]$, $\xi^\theta >0$) follows by a verification argument similar to that of Theorem \ref{VerThm} and an existence result of solution to the BSDE \eqref{bsdeC}. Let $(R, \Theta^R, M) \in \mathcal{L}^\b \times \widehat{\mathcal{L}}\times \mathcal{L}^{\perp}$ be a solution to BSDE \eqref{bsdeC}. We follow the same lines as in the proof of Theorem \ref{VerThm}, with the key difference that \eqref{nuova} takes the form:
 \begin{equation*}\label{nuova1}
\overline{Y}^{\mathbf 0}_t -   \overline{Y}^{\mathbf u}_t= \int_0^t e^{-rs} c_1(u_s^\a) \ud s,
\end{equation*}
and, as a consequence, $\widetilde M_t^{\mathbf u}$ in \eqref{Mu} reduces to 
\begin{align*}\label{Mu1}
\ud \widetilde M_t^{\mathbf u} = &e^{  \eta (\overline{Y}^{\mathbf 0}_{t-} - \overline{Y}^{\mathbf u}_{t-})e^{rT}} \left(\int_0^{+ \infty} \Theta^{R} (t, z) \widetilde m^{\mathbf u}(\ud t, \ud z) + \ud M_t\right).
\end{align*}
As in the proof of Theorem \ref{VerThm} we can show that $\widetilde M^{\mathbf u}$ is an $(\bF, \P^{\mathbf u})$-martingale and that, for any maximizer $u^{(1),*}$  of $ f^\a$ in \eqref{ftildeC},  $\widetilde M^{u^{(1),*}}$ is an $(\bF, \P^{u^{(1),*}})$-martingale.  Consequently, applying Proposition \ref{VT1} (with condition $(ii)$ modified as $D_T= e^{-\eta \xi^\theta}$) we obtain that $R$ necessarily coincides with the Snell Envelope $W^{\theta, 0}$ and $u^{(1),*}$  yields an optimal admissible self-protection strategy.
The existence of a solution to BSDE \eqref{bsdeC} follows once again by applying \cite[Theorem 3.5]{PPS2018}, see for details Appendix A in \cite{BCCS_2025}.  
Combining the verification result with the existence of a solution to \eqref{bsdeC} yields the result. \end{proof}

\begin{remark}\label{REweak2}
As in Remark \ref{REweak1} we can observe that Theorem \ref{SELF} remains valid if the original conditions in  \eqref{GEN} are replaced by the weaker assumptions \eqref{lambda_cond} and \eqref{Weak}.
\end{remark}

\section{The Markovian case} \label{sec:markov}

From now on, we assume that conditions \eqref{lambda_cond} and \eqref{Weak} are in force. Therefore, by Remark \ref{REWEAK}, Theorem \ref{MainResult} applies. In this section, we discuss a Markovian framework where the underlying claim arrival intensity is a function of an $\mathbb{R}$-valued  càdlàg Markov stochastic process $Y$, which does not have common jump times with the loss process $J$. Specifically, the arrival intensity is given by $\lambda^0_t = \Gamma (t,Y_{t^-})$, $t \in [0,T]$, where $\Gamma (t,y)$ is a strictly positive measurable function. 
This framework encompasses  Example (i) in Subsection \ref{Exa}, as well as the Cox shot-noise intensity case, which arises in Example (ii) of Subsection \ref{Exa} by taking $l(z) = 0$ in equation \eqref{intensity_int} and setting $\Gamma(t, y) = y$. 
The Markovian setting and, the assumption of no common jump times between the loss process and the arrival intensity of the claim, allow us to characterize optimal prevention strategies as minimizer of an Hamiltonian that does not depend on the solution of the BSDE in equation \eqref{bsde1}, as shown in the following proposition.
\begin{proposition}\label{Mark}
    Let $\lambda^0_t = \Gamma (t,Y_{t^-})$ where $Y$ is a càdlàg $\mathbb{R}$-valued and  $(\bF, \P^0)$-Markov process, with no common jump times  with the loss process $J$. 
    %Assume the conditions of Theorem \ref{MainResult} hold.
    Let $\mathbf u^*(t,y) = (u^{\a,*}(t,y), u^{\b,*}(t,y))$ be a minimizer of 
    \begin{equation}\label{psiu}
    \begin{split}
\psi^{\mathbf u}(t,y) 
&  =  \eta e^{r(T-t)} (c_1(u^\a) + c_2(u^\b)) + \\
& + \gamma^\a(u^\a) \Gamma(t,y) \int_0^{+\infty} ( e^{\eta e^{r(T-t)}\gamma^\b(u^\b) z} -1)  F(\ud z).
\end{split}
\end{equation}
Then, $\mathbf u^*_t = \mathbf u^*(t,Y_{t^-})$ is an optimal (Markovian) prevention strategy.
\end{proposition}
\begin{proof}
For any Markovian control we have that the pair $(X^\mathbf u, Y)$ is an $(\bF, \P^\mathbf u)$-Markov processes.  Thus, the value process in Equation \eqref{eqn:V} can be written as $V_t = \varphi(t, Y_t)$ where
\begin{equation}
    \label{varphi1}
\varphi(t, y)=\inf_{\mathbf u \in\mathcal{U}_t}\mathbb{E}_{t,y}^{\P^{\mathbf u}}\left[e^{\eta\int_t^T e^{r(T-s)} \left( c_1(u^\a_s)+c_2(u^\b_s) \right) \,\ud s
+\eta\int_t^T\int_0^{+\infty} e^{r(T-s)} \gamma^\b(u^\b_s) z  \,m(\ud s,\ud z)}\right],
\end{equation}
with $\mathbb{E}_{t,y}^{\P^{\mathbf u}}[\cdot]$ denoting the expectation under $\P^{\mathbf u}$
and such that the stochastic factor $\{Y_s,\ s\in [t,T]\}$ satisfies the initial condition $Y_t=y$.
Recalling \eqref{eqn:VisWI} we get that $W^{\mathbf 0}_t=
\varphi(t, Y_t) e^{-\eta \overline{Y}^{\mathbf 0}_t e^{rT}}$, for any  $t \in [0,T]$, and so 
\begin{equation}\label{theta}
\Theta(T_n, Z_n)\!= \!W^{\mathbf 0}_{T_n} - W^{\mathbf 0}_{T_n^-} = e^{-\eta \bar Y^\mathbf 0_{T_n^-}e^{rT}} \!\!\left(\varphi (T_n, Y_{T_n}) e^{\eta Z_n e^{r(T-T_n)}} -  \varphi(T_n^-, Y_{T_n^-})\right).  
\end{equation}
Since $Y$ does not jump at any $T_n$ we have that  $\varphi (T_n, Y_{T_n}) = \varphi(T_n^-, Y_{T_n^-})$ and, by similar considerations as in  \cite[Remark 3.4]{BC_2020} and  in \cite[Remark 7]{ceci-cretarola2025}, we get that
$$\Theta(t,z) = e^{-\eta \overline{Y}^{\mathbf 0}_{t^-} e^{rT}} \varphi(t,Y_{t^-}) 
 (e^{\eta z e^{r(T-t)}} -1) \quad F(\ud z) \ud t \ud \P^0-{\rm a.e.}$$
 Finally, the proof follows from Theorem \ref{MainResult} observing that 
 \begin{align*}
&f(t, \mathbf {u}_t) = \varphi(t,Y_{t^-}) e^{-\eta \overline{Y}^{\mathbf 0}_{t^-} e^{rT}} \int_0^{+\infty} 
 (e^{\eta z e^{r(T-t)}} -1) \Gamma(t, Y_{t^-}) F(\ud z)  \\
& - e^{-\eta \overline{Y}^{\mathbf 0}_t e^{rT}} \varphi(t, Y_{t^-})\Big\{ \eta e^{\eta r (T-t)}(c_1(u_t^\a) + c_2(u_t^\b))\\
& \qquad + \gamma^\a(u_t^\a) \lambda^0_{t^-} \int_0^{+\infty} \Big( e^{\eta e^{r(T-t)}\gamma^\b(u^\b_t) z} -1 \Big )  F(\ud z) \Big\} \\
& \quad \quad \ \ = \varphi(t,Y_{t^-}) e^{-\eta \overline{Y}^{\mathbf 0}_{t^-} e^{rT}} \Big\{ \int_0^{+\infty} 
 (e^{\eta z e^{r(T-t)}} -1) \Gamma(t, Y_{t^-}) F(\ud z)  - \psi^{\mathbf u}(t,Y_{t^-})\Big\}.
    \end{align*}
    \end{proof}

\subsection{The case of constant underlying claim arrival intensity}

In this subsection we discuss the special case of constant underlying claim arrival intensity, i.e.
$\lambda_t^0=\lambda^0 \in \R^+$, for every $t \in [0,T]$. Here, we find an explicit solution to the BSDE \eqref{bsde1} and hence an explicit expression of the value function. Moreover, we show that optimal prevention strategies are deterministic function of time.
\begin{proposition}
Let $\lambda_t^0=\lambda^0 \in \R^+$, for every $t \in [0,T]$.
Then, the value function is given by
\begin{equation}\label{V}\varphi(t) = e^{\int_t^T \inf_{\mathbf u \in U} \psi^\mathbf u(s) \ud s},
\end{equation}
where 
\begin{equation}\label{eq:psi}
\begin{split}
 \psi^{\mathbf u}(t) = & \eta e^{r(T-t)}(c_1(u^\a)+c_2(u^\b))+ \\
 & \lambda^0 \gamma^\a(u^\a)\int_0^{+\infty}\left(e^{\eta \gamma^\b(u^\b)ze^{r(T-t)}}-1\right)F(\ud z).  
 \end{split}
\end{equation}
Moreover, any $\{\mathbf u^*(t), t\in [0,T]\}$ such that 
\begin{equation} \label{psi} \psi^{\mathbf u^*}(t) = \inf_{\mathbf u \in U} \psi^\mathbf u(t), \quad t \in [0,T]
\end{equation}
provides an  optimal (deterministic) prevention strategy.
\end{proposition}
\begin{proof} 
When the claim arrival intensity under $\P^0$ is constant, we get that $V_t= \varphi(t)$, with $\varphi(t)$ being a deterministic function of time to be determined. We look for $\varphi(t)$ differentiable, then by It\^o's product rule
\begin{equation*}
\ud W^0_t = \ud (e^{-\eta \overline{Y}^{\mathbf 0}_t e^{rT}} \varphi(t)) = \varphi'(t) e^{-\eta \overline{Y}^{\mathbf 0}_t e^{rT}}\ud t + \varphi(t) \ud(e^{-\eta \overline{Y}^{\mathbf 0}_t e^{rT}}).
\end{equation*}
Recalling that $\overline{Y}^{\mathbf 0}_t= - \int_0^t \int_0^{+\infty} e^{-rs} z m(\ud s, \ud z)$, we get 
\begin{equation}\label{1}
\ud W^0_t = e^{-\eta \overline{Y}^{\mathbf 0}_t e^{rT}} \varphi'(t) \ud t +  e^{-\eta \overline{Y}^{\mathbf 0}_{t^-} e^{rT}} \varphi(t) 
\int_0^{+\infty} (e^{\eta z e^{r(T-t)}} -1) m(\ud t, \ud z). 
\end{equation}
On the other hand by Theorem \ref{MainResult}
\begin{equation}\label{2}
\ud W^0_t = \int_0^{+ \infty} \Theta (t, z) \widetilde m^{0}(\ud t, \ud z) + \ud M_t + \esssup_{\mathbf u \in \mathcal U}f ( t, W^0_{t^-} , \Theta(t,\cdot),\mathbf u_t)\,\ud t.
\end{equation}
 From \eqref{theta} we have that
 \begin{equation*}
\Theta (t, z) = e^{-\eta \overline{Y}^{\mathbf 0}_{t^-} e^{rT}} \varphi(t) 
 (e^{\eta z e^{r(T-t)}} -1) 
\end{equation*}
which plugging in \eqref{1} implies
\begin{equation}\label{3}
\ud W^0_t = e^{-\eta \overline{Y}^{\mathbf 0}_t e^{rT}} \varphi'(t) \ud t +  
\int_0^{+\infty} \Theta(t,z) m(\ud t, \ud z). 
\end{equation}
We observe that
\begin{align*}
f(t,\mathbf u_t) = & \int_0^{+\infty} \Theta(t,z) \lambda^0 F(\ud z) 
- e^{-\eta \overline{Y}^{\mathbf 0}_t e^{rT}} \varphi(t)\Big [\eta e^{\eta r (T-t)}(c_1(u_t^\a) + c_2(u_t^\b)) -  
\\
& + \gamma^\a(u_t^\a) \lambda^0\int_0^{+\infty}  (e^{\eta e^{r(T-t)}\gamma^\b(u^\b_t) z} -1 )  F(\ud z) \Big],
\end{align*}
%\end{equation}
and so comparing \eqref{3} with \eqref{2} we get that $M_t=0$, for any $t\in[0,T]$, and $\varphi(t)$ has to solve the following ODE
\begin{equation}\label{ODE}
    \varphi'(t) + \varphi(t) \inf_{\mathbf u  \in [0,\zeta_1] \times [0,\zeta_2]} \psi^{\mathbf u}(t)=0, \quad \forall t \in [0,T),
\end{equation}
with final condition $\varphi(T)=1$,
whose solution is given by \eqref{V}. 
We observe that under the continuity assumption on the functions $c_i(u^{(i)})$, $\gamma^{(i)}(u^{(i)})$, $i=1,2$,  the existence of a minimizer in \eqref{ODE} is ensured. Finally, again by Theorem \ref{MainResult}, since every deterministic function of time $\mathbf u^*(t)$ that satisfies \eqref{psi} realizes the essential supremum in \eqref{bsde1}, the proof is concluded. 
\qedsymbol{}
\end{proof}

\begin{remark}
    We observe that in the case of constant underlying claim arrival intensity, we find an explicit solution $(W^{\mathbf 0}, \Theta, M)$ to the BSDE \eqref{bsde1}, which is given by $$
    W^{\mathbf 0}_t=
\varphi(t) e^{-\eta \overline{Y}^{\mathbf 0}_t e^{rT}},\ \Theta (t, z) = e^{-\eta \overline{Y}^{\mathbf 0}_{t^-} e^{rT}} \varphi(t) 
 (e^{\eta z e^{r(T-t)}} -1),\ M_t=0, \quad t \in [0,T]$$ with $\varphi(t)$ given in \eqref{V}.
 \end{remark}

In the next proposition, we focus on the case where the claim arrival intensity is bounded, and we provide sufficient conditions for the convexity of $\psi^{\mathbf u}(t,y)$ as defined in \eqref{psiu}. These conditions enable us to characterize the optimal prevention strategy via first-order conditions. For simplicity, we assume a linear self-insurance impact function given by $\gamma^{\b}(u^\b) = 1 - u^\b$, for $u^\b \in [0,1]$ (i.e., $\zeta_2 = 1$), and we denote the self-protection impact function $\gamma^\a(u^\a)$ simply by $\gamma(u^\a)$.
%suitable assumptions on the cost and impact functions we can characterize the optimal prevention strategy in terms of the first order conditions of $\psi^{\mathbf u}$ in terms of the first order conditions. 

\begin{proposition}
    Let us assume 
    \begin{itemize}
        \item [(i)]
    bounded intensity: $\Gamma(t,y) \leq \Lambda \in \R^+$, for any $(t,y) \in [0,T] \times \R$;
    
    \item [(ii)] convexity of the cost functions and the self-protection impact function: 
    $$c_1''(u^\a) \ge 0,\ \gamma''(u^\a) > 0, \quad \forall u^\a \in [0,\zeta_1], \quad c_2''(u^\b) \ge 0,\ \forall u^\b \in [0,1];
    $$  
     \item [(iii)] log-convexity condition of the self-protection impact function:
    \begin{equation}\label{gamma_conv}
        \frac{{\gamma}''(u^\a)}{\gamma(u^\a)} \ge \left(\frac{{\gamma}'(u^\a)}{\gamma(u^\a)}\right)^2, \quad \forall u^\a \in [0,\zeta_1],
   \end{equation}
    \item [(iv)] growth bound on ${\gamma}''(u^\a)$:
\begin{equation}\label{gamma_2der_bound}
     {\gamma}''(u^\a) \leq \frac{\eta}{\Lambda} c_1''(u^\a), \quad \forall u^\a \in [0, \zeta_1].  
   \end{equation}
    \end{itemize}
       Then, for every $(t,y) \in [0,T]\times \mathbb{R}$, the function $\psi^{\mathbf u}(t,y)$ in \eqref{psiu} is strictly convex in $\mathbf u \in U=[0,\zeta_1] \times [0,\zeta_2]$  and there exists a unique optimal prevention strategy $\mathbf{u}^*_t= \mathbf{u}^*(t,Y_{t^-})$ where $\mathbf{u}^*(t,y)=(u^{\a,*}(t,y), u^{\b,*}(t,y))$, with 
    $u^{(i),*}(t,y) = \min\{ \max\{ 0, \bar u^{(i)}(t,y)\}, \zeta_i\}$, $i=1,2$
    being $\bar{\mathbf u}(t,y) = (\bar u^\a(t,y), \bar u^\b(t,y))$ the unique solution to the first order conditions 
    \begin{align*}
\frac{\partial\psi^{\mathbf u}}{\partial u^\a}(t,y)= & \eta e^{r(T-t)}c_1'(u^\a)+\Gamma(t,y) {\gamma}'(u^\a) \int_0^{+\infty}\!\!\!\left(e^{\eta(1-u^\b)ze^{r(T-t)}}-1\right)F(\ud z)=0, \\%\label{der1}\\
  \frac{\partial\psi^{\mathbf u}}{\partial u^\b}(t,y)
   = & \eta e^{r(T-t)}c_2'(u^\b) + \\ & - \eta e^{r(T-t)} \Gamma(t,y) {\gamma}(u^\a)\int_0^{+\infty}\!\!\left( z e^{\eta(1-u^\b)ze^{r(T-t)}}-1\right)F(\ud z)=0. %\label{der2}
    \end{align*}
    \end{proposition}
\begin{proof}
To prove the result, we show that the Hessian matrix of the function $\psi^{\mathbf u}(t,y)$ is positive definite, for each $(t,y) \in [0,T]\times \mathbb{R}$. For every $(t,y) \in [0,T]\times \mathbb{R}$, the second order derivatives of $\psi^{\mathbf u}(t,y)$ are given by
\begin{align}
    \frac{\partial^2\psi^{\mathbf u}}{\partial (u^\b)^2}(t,y)&=\eta e^{r(T-t)}c_2''(u^\b)+\eta^2 e^{2r(T-t)} \Gamma(t,y) \gamma (u^\a) \int_0^{+\infty}\!\!\! z^2 e^{\eta(1-u^\b)ze^{r(T-t)}}F(\ud z), \label{der_psi_u}\\
  \frac{\partial^2\psi^{\mathbf u }}{\partial (u^\a)^2}(t,y)&= \eta e^{r(T-t)}c_1''(u^\a)+ \Gamma(t,y) \gamma''(u^\a) \int_0^{+\infty}\left(e^{\eta(1-u^\b)ze^{r(T-t)}}-1\right)F(\ud z), \label{der_psi_e}\\
 \frac{\partial^2\psi^{\mathbf u}}{\partial u^\b \partial u^\a}(t,y)&= \frac{\partial^2\psi^{\mathbf u,}}{\partial u^\a \partial u^\b}(t,y)=-\Gamma(t,y) \gamma'(u^\a) \int_0^{+\infty}\!\!\!\!\eta z e^{r(T-t)} e^{\eta(1-u^\b)ze^{r(T-t)}}\!F(\ud z).
\end{align}   
Note that \eqref{der_psi_u} and \eqref{der_psi_e} are positive, for every $(t,y) \in [0,T]\times \mathbb{R}$. Moreover, by the Cauchy-Schwarz inequality combined with \eqref{gamma_conv} and \eqref{gamma_2der_bound}, we have that
\begin{align*}
  &\left( \frac{\partial^2\psi^{\mathbf u}}{\partial u^\a \partial u^\b}(t,y)\right)^2\\
  &= (\gamma'(u^\a))^2 \left(\int_0^{+\infty}\!\!\!z\eta e^{r(T-t)}e^{\frac{1}{2}\eta(1-u^\b)ze^{r(T-t)}} \Gamma^{\frac{1}{2}}(t,y)\cdot e^{\frac{1}{2}\eta(1-u^\b)ze^{r(T-t)}}\Gamma^{\frac{1}{2}}(t,y) F(\ud z)  \right)^2\\
  & \leq \gamma''(u^\a) \gamma(u^\a)\int_0^{+\infty}z^2\eta^2 e^{2r(T-t)}e^{\eta(1-u^\b)ze^{r(T-t)}}\Gamma(t,y) F(\ud z) \\
  & \qquad \qquad \cdot \int_0^{+\infty}e^{\eta(1-u^\b)z e^{r(T-t)}}\Gamma(t,y) F(\ud z)\\
  & \leq  \frac{\partial^2\psi^{\mathbf u}}{\partial (u^\b)^2}(t,y) \cdot \int_0^{+\infty}e^{\eta(1-u)ze^{r(T-t)}}\gamma''(u^\a) \Gamma(t,y) F(\ud z)\\
  & \leq \frac{\partial^2\psi^{\mathbf u}}{\partial (u^\b)^2}(t,y) \cdot \frac{\partial^2\psi^{\mathbf u}}{\partial (u^\a)^2}(t,y), \quad \forall (t,y) \in [0,T]\times \mathbb{R},
\end{align*}
which implies that the determinant of the Hessian matrix of the function $\psi^{\mathbf u}(t,y)$ is positive, yielding that the Hessian matrix is positive definite, for every $(t,y) \in [0,T]\times \mathbb{R}$.
The proof is completed by applying  Proposition \ref{Mark}.
\end{proof}
\begin{remark}
The log-convexity condition \eqref{gamma_conv} not only shapes the behavior of \( \gamma(u^\a) \) 
but also establishes a stronger and more precise relationship between its first and second derivatives. 
Conversely, the growth bound \eqref{gamma_2der_bound} ensures that \( \gamma''(u^\a) \) remains controlled, 
effectively bounding the curvature of \( \gamma(u^\a) \) in relation to \( c_1''(u^\a) \).
\end{remark}

\begin{example}
   We show some examples of functions \( \gamma(x) \) and \( c_1(x) \) satisfying all required conditions.
To ensure all conditions, including the growth bound, we consider the following choices:
\begin{itemize}
\item[(i)] 
$\gamma(x) = e^{-\alpha x}, \quad c_1(x) = x^2, \quad \text{with} \quad 0<\alpha < \sqrt{\frac{2 \eta }{\Lambda}}, \quad x \in [0, \zeta_1]$.
\item[(ii)] Assume $\Lambda < \eta$,
\[
\gamma(x) = \frac{1}{1+x}, \quad c_1(x) = (1+x)^2-1, \quad \text{with} \quad x \in [0, \zeta_1].
\]
\end{itemize}
\end{example}

\section*{Acknowledgements}
The authors are members of Gruppo Nazionale per l’Analisi Matematica, la Probabilità e le loro
Applicazioni (GNAMPA) of Istituto Nazionale di Alta Matematica (INdAM). The first author was partially supported by European Union-Next Generation EU - PRIN research project n. 2022BE-MMLZ and by European Union-Next Generation EU - PRIN PNRR CUP B53D23027770001 - project n. P20225SP685. The second author was partially supported by  the European Union-Next Generation EU - PRIN research project n. 2022FPLY97.

\bigskip

\bibliographystyle{plainnat}
\bibliography{references_CC}

\appendix
\section{Proof of Theorem  \ref{thm:bsde}} 
\label{app:proof}

\noindent We provide here the detailed proof of Theorem \ref{thm:bsde}, which establishes the existence of a solution to the BSDE  \eqref{bsde} under the stated assumptions.

\begin{proof} 
We verify that conditions $\mathbf{(F1)}$-$\mathbf{(F5)}$ in  \cite[Theorem 3.5]{PPS2018} are satisfied.  
\begin{itemize}
\item[$\mathbf{(F1)}$] The process $\{C_t := \int_0^t\int_0^{+\infty} z \widetilde m^0(\ud s, \ud z),\ t \in [0,T]\}$ is an $(\bF, \P^0)$-martingale such that $\sup_{t\in [0,T]}\mathbb{E}^{\P^0}[C^2_t] < +\infty.$ Indeed, for any $t \in [0,T]$
$$\mathbb{E}^{\P^0}[C^2_t] = \mathbb{E}^{\P^0}\left [\int_0^t\int_0^{+\infty} z^2 \lambda^0_t F(\ud z) \ud t \right] \leq \mathbb{E}^{\P^0}[Z^2] \mathbb{E}^{\P^0}\left [\int_0^T\lambda^0_t \ud t \right ]< +\infty.
$$
The disintegration property is satisfied because $K^\omega_t(\ud z) = \lambda^0_{t}(\omega) F(\ud z).$ Here, $K^{\omega}$ denotes the transition kernel
on $(\Omega \times [0,T], {\mathcal P})$, with ${\mathcal P}$ representing the $\bF$-predictable sigma-algebra on $\Omega \times [0,T]$.

\item[$\mathbf{(F2)}$] For a given $\widehat \beta>0$, the terminal condition of the BSDE satisfies $\mathbb E^{\P^0} \left[ e^{\widehat \beta A_T} \Xi^2 \right]<+ \infty.$ The process $A=\{A_t,\ t \in [0,T]\}$ is defined in \eqref{A} below and $A_T \leq \int_0^T 4 \lambda^0_t \ud t+ \bar K$, with $\bar K>0$ being a suitable constant (see ($\mathbf{F4}$) for details). Indeed, exploiting the trivial relation $ab \leq \frac{1}{2} (a^2 + b^2)$, for any $a,b \in \mathbb{R}$, in view of \eqref{GEN}, we get
%$$\mathbb{E}^{\P^0}[\Xi^2 e^{\hat \beta A_T}] = \mathbb{E}^{\P^0}[e^{- \eta a Y^{\mathbf 0}_T}] \leq \mathbb{E}^{\P^0}[e^{\eta a e^{rT} J_T}] < \infty.$$ 
\begin{align*}
\mathbb E^{\P^0} \left[ e^{\widehat \beta A_T} \Xi^2 \right] 
&\le \frac{1}{2}\mathbb E^{\P^0}\left[ e^{ 2 \widehat \beta A_T}\right] + \frac{1}{2}\mathbb E^{\P^0} \left[\Xi^4 \right] \\
&\leq K \mathbb E^{\P^0} \left[ e^{ 8\widehat \beta \int_0^T  \lambda^0_{s}\ud s }\right] +  \frac{1}{2}\mathbb E^{\P^0} \left[e^{4\eta e^{rT} J_T} \right] < + \infty 
\end{align*} 
for a suitable constant $K>0$.

\item[$\mathbf{(F3)}$] The generator 
%of BSDE \eqref{bsde} satisfies a stochastic Lipschitz condition. We recall that
$F(t, \omega, y, \theta  ( \cdot)):= \esssup_{\mathbf{u} \in \mathcal{U}} f ( t, \omega, y, \theta  ( \cdot), \mathbf u_t )$ defined on the space $$\mathbb{M} = \{ (t,\omega,y,\theta(\cdot)): (t,\omega,y)\in[0,T]\times\Omega\times (0,+\infty),\,
    \theta(\cdot): [0,+\infty)\to\mathbb{R}\ \textrm{measurable} \}$$
satisfies a stochastic Lipschitz condition, i.e., there exist two positive $\mathbb{F}$-predictable processes $\gamma=\{\gamma_t,\ t \in [0,T]\}$, $\bar{\gamma}=\{\bar \gamma_t,\ t \in [0,T]\}$ such that 
\begin{equation}\label{eqn;stochlip}
\big| F(t, \omega, y, \theta  ( \cdot) ) - F(t, \omega, y', \theta' (\cdot) ) \big| ^2
	\leq \gamma_t (\omega ) |y - y'|^2
	+\bar{\gamma}_t (\omega ) \left( ||| \theta ( \cdot) - \theta' ( \cdot) |||_t (\omega ) \right)^2 ,
\end{equation}
where:
\begin{equation*}
\begin{split}
\left( ||| \theta ( \cdot) |||_t (\omega ) \right)^2 := \int_0^t \theta^2(z) K_t^\omega(\ud z) = \int_0^{+\infty} \theta^2(z) \lambda^0_{t}(\omega) F(\ud z).
\end{split}
\end{equation*}
Exploiting the definition of $F$, we need to deal with the $\P^0$-essential supremum:
\begin{align*}
&\bigl| F (t, \omega, y, \theta  ( \cdot,) ) - F(t, \omega, y', \theta' (\cdot )) \bigr| ^2 \\
	& \qquad \qquad \leq \left( \esssup_{\mathbf{u} \in \mathcal{U}} \bigl| f ( t, \omega, y, \theta  ( \cdot), \mathbf u_t(\omega) ) -   f ( t, \omega, y', \theta'  ( \cdot), \mathbf u_t(\omega) ) \bigr| \right)^2,
\end{align*}
and we preliminarily compute for any $\mathbf u \in \mathcal{U}$ 
\begin{align}
& \bigl| f( t, y, \theta  ( \cdot), \mathbf u_t  ) -  f( t, y', \theta'  ( \cdot), \mathbf u_t) \bigr| \nonumber\\
& = \big| (y- y') [\lambda^0_{t^-} + \eta e^{r (T-t)} (c_1(
u^\a_t) + c_2(u^\b_t))]+  \int_{0}^{+\infty} \!\!\! \left( \theta(z) - \theta'(z)\right)  \gamma^\a(u_t^\a) \lambda^0_{t} F(\ud z) \big| \nonumber\\
& \le  \bigl| y- y' \bigr| (\lambda^0_{t^-} + \bar k)  +  2 \lambda^0_{t^-} \int_{0}^{+\infty} \left| \theta(z) - \theta'(z) \right| F (\ud z),\label{nuova2}
\end{align}
\noindent where
\begin{equation*}\label{bar k}
\bar k = \eta e^{rT}(c_1(\zeta_1) + c_2(\zeta_2)).
\end{equation*}

\noindent Now, from \eqref{nuova2}  and  the relation $ (a+ b )^2 \leq 2 (a^2 + b^2)$, with $a,b \in \R$, we find:
\begin{align*}
&\bigl| F (t, \omega, y, \theta  ( \cdot,) ) - F(t, \omega, y', \theta' (\cdot )) \bigr| ^2 \\
& \qquad \qquad \leq 2 \bigl| y- y' \bigr|^2 (\lambda^0_{t}(\omega) + \bar k)^2  +  2 \left(\int_{0}^{+\infty} \!\! \left| \theta(z) - \theta'(z) \right| \lambda^0_{t}(\omega) F (\ud z)\right)^2.
\end{align*}
\noindent Finally, we apply Jensen's inequality and we get
that \eqref{eqn;stochlip} is satisfied with
$\gamma_t=2(\lambda^0_{t^-} + \bar k)^2$ and $\bar{\gamma}_t=2\lambda^0_{t}$.
\item[$\mathbf{(F4)}$] The process $A$ defined as
\begin{equation}\label{A}
A_t := \int_{0}^{t} \alpha _{s}^2 \,ds, \quad \alpha_s^2 =  \max\{ \sqrt{\gamma_t}, \bar \gamma_t\}
\end{equation}
is such that for any $t\in [0,T]$, $A_t    \leq 2 \int_{0}^{T} \lambda^0_s \,ds + \bar K$, with $\bar K>0$ suitable constant and $\Delta A_t \leq \Phi, \P^0-$a.s. holds true for any $\Phi >0$ since $A$ has no jumps.

\item[$\mathbf{(F5)}$] For a given $\widehat\beta >0$  we have that
\begin{equation*}
\mathbb{E}^{\P^0} \left[  \int_0^T e^{\widehat\beta A_t} \frac{|F (t, 0, 0) |^2}{\alpha_t^2}  \ud t \right] < +\infty,
\end{equation*}
because $F (t,0,0) =0$.
\end{itemize}
It now remains to prove that there exists $\widehat \beta>0$ such that  
\begin{equation}\label{Mphi}
M^{\Phi}(\widehat \beta) = \frac{9}{\widehat\beta} + \frac{\Phi^2 (2 + 9 \widehat \beta)}{\sqrt{\widehat \beta^2 \Phi^2 +4} -2} \exp{\left( \frac{\widehat \beta \Phi + 2 -\sqrt{\widehat \beta^2 \Phi^2 +4}}{2} \right)} < \frac12,
\end{equation}
with $\Phi>0$ introduced in $\mathbf{(F4)}$. 
Thanks to \cite[Lemma 3.4]{PPS2018}, since $\lim_{\widehat \beta \rightarrow \infty} M^\Phi(\widehat \beta) = 9 e \Phi$, it suffices to take $\Phi < \frac{1}{18 e}$, in which case  $M^{\Phi}(\widehat \beta)<\frac12$ for $\widehat \beta$ sufficiently large.
According to \cite[Theorem 3.5]{PPS2018} there exists a solution  $(R,\Theta^{R}, M^R)\in \mathcal{L}^\b \times \widehat{\mathcal{L}} \times \mathcal{L}^{\perp}$  to BSDE \eqref{bsde1}, and this concludes the proof.
\qedsymbol{}
\end{proof}
\end{document}